\documentclass[a4paper,reqno,10pt]{amsart}

\usepackage{epsf,graphicx,epsfig}
\usepackage{amscd}
\usepackage{amsmath,latexsym,amssymb,amsthm}
\usepackage[nospace,noadjust]{cite}
\usepackage{textcomp}
\usepackage{setspace,cite}
\usepackage{lscape,fancyhdr,fancybox}
\usepackage{stmaryrd}
\usepackage[all,cmtip]{xy}
\usepackage{tikz}
\usepackage{cancel}
\usetikzlibrary{shapes,arrows,decorations.markings}
\usepackage{graphicx}

\usepackage{color}
\usepackage{url}
\usepackage{enumerate}
\usepackage[mathscr]{euscript}
  \theoremstyle{plain}

\swapnumbers
    \newtheorem{thm}{Theorem}[section]
    \newtheorem{prop}[thm]{Proposition}

    \newtheorem{subsec}[thm]{}
\theoremstyle{definition}
    \newtheorem{defn}[thm]{Definition}
        \newtheorem{remark}[thm]{Remark}
    \newtheorem{exam}[thm]{Example}

    \newtheorem{notation}[thm]{Notation}

\theoremstyle{remark}

\title{}
\author{}
\date{}
\usepackage{amssymb}

\usepackage{hyperref}
\hypersetup{
	colorlinks,
	citecolor=blue,
	filecolor=black,
	linkcolor=blue,
	urlcolor=black
}

\begin{document}
\title{Relative Rota-Baxter Leibniz algebras, their characterization and cohomology}

\author{Apurba Das}
\address{Department of Mathematics,
Indian Institute of Technology, Kharagpur 721302, West Bengal, India.}
\email{apurbadas348@gmail.com}

%\author{Satyendra Kumar Mishra}
%\address{Statistics and Mathematics Unit, Indian Statistical Institute Bangalore Centre, Bangalore-560059, Karnataka, India.}
%\email{satyamsr10@gmail.com}

%\curraddr{}
%\email{}


\begin{abstract}
Recently, relative Rota-Baxter (Lie/associative) algebras are extensively studied in the literature from cohomological points of view. In this paper, we consider relative Rota-Baxter Leibniz algebras (rRB Leibniz algebras) as the object of our study. We construct an $L_\infty$-algebra that characterizes rRB Leibniz algebras as its Maurer-Cartan elements. Then we define representations of an rRB Leibniz algebra and introduce cohomology with coefficients in a representation. As applications of cohomology, we study deformations and abelian extensions of rRB Leibniz algebras. 
%Finally, we define homotopy rRB Leibniz algebras and classify some particular classes of such homotopy algebras in terms of the cohomology of rRB Leibniz algebras.
\end{abstract}

\maketitle

%\author{{}\footnote{{\em Acknowledgements.} Currently, A. Das is not affiliated with any Institute. However, some parts of the work was carried out when he was a postdoctoral fellow at IIT Kanpur, India. The work of S. K. Mishra is supported by the NBHM postdoctoral fellowship.}}

%\author{{}\footnote{{\em Addresses.} (A. Das) Email: apurbadas348@gmail.com\\(S. K. Mishra) Statistics and Mathematics Unit, Indian Statistical Institute Bangalore Centre, Bangalore-560059, Karnataka, India. Email: satyamsr10@gmail.com}}

\noindent

\thispagestyle{empty}

\tableofcontents

\vspace{0.2cm}

\section{Introduction}

%\textcolor{red}{QUES: Relation with Poisson bimodule (if any)}

%\textcolor{red}{QUES: Some relations with infinitesimal bialgebras}

Rota-Baxter operators appeared in the work of Baxter \cite{baxter} in the fluctuation theory of probability. Such operators on associative algebras were subsequently studied by Rota \cite{rota}, Atkinson \cite{atkinson}, Miller \cite{miller}, Cartier \cite{cartier} among others. In the last twenty years, Rota-Baxter operators were used to study Yang-Baxter equations, infinitesimal bialgebras, dendriform algebras and double algebras \cite{aguiar-pre,bai-splitting,double-algebra}. Connes and Kreimer recently found important applications of Rota-Baxter operators in renormalization of quantum field theories \cite{conn}. In \cite{uchino} Uchino introduced generalized Rota-Baxter operators (mostly known as relative Rota-Baxter operators in the literature) on bimodules over an associative algebra. Rota-Baxter operators and relative Rota-Baxter operators in the context of Lie algebras were first introduced by Kupershmidt in the study of classical $r$-matrices \cite{kuper}. Recently, cohomology and deformation theory of relative Rota-Baxter operators in the context of Lie and associative algebras were developed in \cite{tang,das-rota}. A relative Rota-Baxter Lie (resp. associative) algebra is a triple consisting of a Lie (resp. associative) algebra, a representation and a relative Rota-Baxter operator. In \cite{laza-rota,DasSK} the authors construct a suitable $L_\infty$-algebra which characterize relative Rota-Baxter Lie (resp. associative) algebras as its Maurer-Cartan elements. Subsequently, representations and cohomology of relative Rota-Baxter Lie (resp. associative) algebras are explicitly studied in \cite{jiang-sheng,DasSK2}.

\medskip

Leibniz algebras are non-skewsymmetric analogs of Lie algebras. They first appeared in the work of Bloh \cite{bloh} and were rediscovered by Loday \cite{loday-une}. Leibniz algebras are extensively studied in the literature, see for instance \cite{loday-pira,bala-leib,loday-cup,branes,ayupov}. Recently, Rota-Baxter operators and relative Rota-Baxter operators in the context of Leibniz algebras, and their relations with Leibniz Yang-Baxter equation and Leibniz bialgebras are studied by Sheng and Tang \cite{sheng-tang}. Cohomology and deformation theory of relative Rota-Baxter operators in Leibniz algebras are also considered by Tang, Sheng and Zhou \cite{tang-leibniz}. A relative Rota-Baxter Leibniz algebra (rRB Leibniz algebra in short) is a triple consisting of a Leibniz algebra, a representation and a relative Rota-Baxter operator (see Definition \ref{rrbl}). An rRB Leibniz algebra gives rise to a Leibniz algebra structure on the representation space, which is called the induced Leibniz algebra. 
%The study of rRB Leibniz algebras are not yet considered in the literature.

\medskip

Our aim in this paper is to extensively study rRB Leibniz algebras. At first, using Voronov's derived bracket \cite{voro}, we construct an $L_\infty$-algebra whose Maurer-Cartan elements are precisely rRB Leibniz algebras (cf. Theorem \ref{l-inf-mc}). Given an rRB Leibniz algebra, we also construct a new $L_\infty$-algebra (twisted by the Maurer-Cartan element corresponding to the rRB Leibniz algebra) that governs Maurer-Cartan deformations of the structure (cf. Theorem \ref{twisted-l-inf}). Next, we introduce representations of an rRB Leibniz algebra. In particular, we get representations of Rota-Baxter Leibniz algebras. Any rRB Leibniz algebra can be regarded as a representation of itself, called the adjoint representation. Given any representation, we can construct the dual representation (cf. Proposition \ref{prop-dual}). We show that a representation of an rRB Leibniz algebra gives rise to a representation of the induced Leibniz algebra (cf. Proposition \ref{vr-h}). This representation plays a key role in the description of the cohomology of the rRB Leibniz algebra. We also construct the semidirect product rRB Leibniz algebra from a given representation of a rRB Leibniz algebra (cf. Theorem \ref{thm-semi}). Next, we focus our attention on the cohomology theory of rRB Leibniz algebras. Using the Maurer-Cartan characterization of rRB Leibniz algebra (given in Theorem \ref{l-inf-mc}), we define the cohomology of an rRB Leibniz algebra with coefficients in the adjoint representation. Then we use this construction to define the cohomology with coefficients in any arbitrary representation.

\medskip

It is well-known that \cite{gers,weibel} cohomology of algebraic structure is closely related to deformations and abelian extensions. We show that our cohomology of rRB Leibniz algebras is suitable to study deformations and abelian extensions of rRB Leibniz algebras. We study $\mathsf{R}$-deformations of an rRB Leibniz algebra, where $\mathsf{R}$ is an augmented unital ring. In particular, we get formal deformations and infinitesimal deformations. Among others, we show that the set of equivalence classes of infinitesimal deformations of an rRB Leibniz algebra has a one-to-one correspondence with the second cohomology group of the rRB Leibniz algebra with coefficients in the adjoint representation (cf. Theorem \ref{inf-def-2}). We also define abelian extensions of an rRB Leibniz algebra and show that isomorphism classes of such abelian extensions are classified by the second cohomology group with coefficients in a representation (cf. Theorem \ref{abelian-ext-2}).

\medskip

The paper is organized as follows. In Section \ref{sec-2}, we recall some preliminaries on Leibniz algebras, relative Rota-Baxter operators and $L_\infty$-algebras. In Section \ref{sec-3}, we consider rRB Leibniz algebras and construct the $L_\infty$-algebra that characterizes rRB Leibniz algebras as Maurer-Cartan elements. Representations of rRB Leibniz algebras and some basic constructions are given in Section \ref{sec-4}. Cohomology of an rRB Leibniz algebra (with coefficients in adjoint representation and coefficients in arbitrary representation) is given in Section \ref{sec-5}. Finally, in Section \ref{sec-6}, we study deformations and abelian extensions of rRB Leibniz algebras as applications of our cohomology.

\medskip

All vector spaces, linear maps and tensor products are over a field ${\bf k}$ of characteristic $0$. A permutation $\sigma \in S_n$ is said to be a $(i,j)$-shuffle with $i+j=n$ if $\sigma (1) < \cdots < \sigma (i)$ and $\sigma(i+1) < \cdots < \sigma (i+j)$.
We denote the set of all $(i,j)$-shuffles by $Sh (i, j)$.

%\begin{notation}

%\end{notation}
 
\section{Preliminaries}\label{sec-2}
In this section, we recall some basic definitions and results on Leibniz algebras, relative Rota-Baxter operators and $L_\infty$-algebras. Our main references are \cite{loday-pira,bala,sheng-tang,lada-markl,laza-rota,voro}.

\begin{defn}
A {\bf Leibniz algebra} is a vector space $\mathfrak{g}$ together with a linear operation (called the bracket) $[~,~]_\mathfrak{g} : \mathfrak{g} \otimes \mathfrak{g} \rightarrow \mathfrak{g}$ satisfying the following Leibniz identity
\begin{align}\label{leib-iden}
[x,[y,z]_\mathfrak{g}]_\mathfrak{g} = [[x,y]_\mathfrak{g}, z ]_\mathfrak{g} + [y, [x, z]_\mathfrak{g} ]_\mathfrak{g}, ~\text{ for } x, y, z \in \mathfrak{g}.
\end{align}
A Leibniz algebra may be simply denoted by the underlying vector space  $\mathfrak{g}$ when the bracket is clear from the context.
\end{defn}

\begin{remark}
The Leibniz algebra considered in the above definition is called a left Leibniz algebra as the identity (\ref{leib-iden}) is equivalent to the fact that the maps $[x, -]_\mathfrak{g} : \mathfrak{g} \rightarrow \mathfrak{g}$ by fixing left coordinate are derivations for the bracket on $\mathfrak{g}$. Therefore, one can also define right Leibniz algebras. If $[~,~]_\mathfrak{g} : \mathfrak{g} \otimes \mathfrak{g} \rightarrow \mathfrak{g}$ is a left Leibniz bracket on a vector space $\mathfrak{g}$, then it is easy to verify that $[x,y]_\mathfrak{g}^\mathrm{op} : = [y, x]_\mathfrak{g}$ is a right Leibniz bracket on $\mathfrak{g}$ and vice-versa. It is important to remark that the results of the present paper can be adapted to right Leibniz algebras by suitable modifications.
\end{remark}

\begin{defn}
Let $\mathfrak{g}$ be a Leibniz algebra. A {\bf representation} of $\mathfrak{g}$ is a vector space $V$ together with linear maps (called the left and right $\mathfrak{g}$-actions, respectively) $l_V : \mathfrak{g} \otimes V \rightarrow V$ and $r_V : V \otimes \mathfrak{g} \rightarrow \mathfrak{g}$ satisfying for $x, y \in \mathfrak{g}$, $v \in V$,
\begin{align*}
l_V (x, l_V (y, v)) =~& l_V ([x, y]_\mathfrak{g}, v) + l_V (y, l_V (x, v)),\\
l_V (x, r_V (v, y)) =~& r_V (l_V(x, v), y) + r_V (v, [x, y]_\mathfrak{g}),\\
r_V (v, [x, y]_\mathfrak{g}) =~& r_V ( r_V (v, x), y) + l_V (x, r_V (v, y)).
\end{align*}
\end{defn}

\begin{exam}
Any Leibniz algebra $\mathfrak{g}$ is a representation of itself with both left and right actions $l_\mathfrak{g}, r_\mathfrak{g} : \mathfrak{g} \otimes \mathfrak{g} \rightarrow \mathfrak{g}$ are given by the Leibniz bracket on $\mathfrak{g}$. This is called the adjoint representation.
\end{exam}

Let $\mathfrak{g}$ be a Leibniz algebra and $V$ be a representation of it. For each $n \geq 0$, we define an abelian group $C^n (\mathfrak{g}, V) := \mathrm{Hom}(\mathfrak{g}^{\otimes n}, V)$ and a map $\delta_{\mathfrak{g}, V} : C^n (\mathfrak{g}, V) \rightarrow C^{n+1} (\mathfrak{g}, V)$ by
\begin{align*}
(\delta_{\mathfrak{g}, V} (f)) (x_1, \ldots, x_{n+1}) =~& \sum_{i=1}^n (-1)^{i+1} ~ l_V (x_i, f (x_1, \ldots, \widehat{x_i}, \ldots, x_{n+1})) + (-1)^{n+1} ~r_V (f(x_1, \ldots, x_n), x_{n+1}) \\
~&+ \sum_{1 \leq i < j \leq n+1} (-1)^i ~ f (x_1, \ldots, \widehat{x_i}, \ldots, x_{j-1}, [x_i, x_j]_\mathfrak{g}, x_{j+1}, \ldots, x_{n+1}),
\end{align*}
for $f \in C^n(\mathfrak{g}, V)$ and $x_1, \ldots, x_{n+1} \in \mathfrak{g}$. Then $\{ C^\bullet (\mathfrak{g}, V), \delta_{\mathfrak{g},V} \}$ is a cochain complex \cite{loday-pira}. The corresponding cohomology groups are called the cohomology of $\mathfrak{g}$ with coefficients in the representation $V$.

\medskip

In \cite{bala} Balavoine introduced a graded Lie algebra associated to a vector space which plays an important role in the study of Leibniz algebras. Let $\mathfrak{g}$ be a vector space (not necessarily a Leibniz algebra). Consider the shifted graded vector space $C^{\bullet + 1} (\mathfrak{g}, \mathfrak{g}) := \oplus_{ n \geq 0} C^{n+1} (\mathfrak{g}, \mathfrak{g})$. Then the {\bf Balavoine bracket}
%\begin{align*}
$[~,~]_\mathsf{B} : C^{\bullet + 1} (\mathfrak{g}, \mathfrak{g}) \otimes C^{\bullet + 1} (\mathfrak{g}, \mathfrak{g}) \rightarrow C^{\bullet + 1} (\mathfrak{g}, \mathfrak{g})$ is a degree $0$ bracket
%\end{align*}
given by
\begin{align*}
&[f, g]_\mathsf{B} (x_1, \ldots, x_{m+n+1}) \\ &:= \sum_{i=1}^{m+1} (-1)^{(i-1)n} \sum_{\sigma \in Sh (i-1, n)} (-1)^\sigma ~ f (x_{\sigma (1)}, \ldots, x_{\sigma (i-1)}, g (x_{\sigma (i)}, \ldots, x_{\sigma (i+n-1)}, x_{i+n}), \ldots, x_{m+n+1}) \\
&- (-1)^{mn} \sum_{i=1}^{n+1} (-1)^{(i-1)m} \sum_{\sigma \in Sh (i-1, m)} (-1)^\sigma ~ g (x_{\sigma (1)}, \ldots, x_{\sigma (i-1)}, f (x_{\sigma (i)}, \ldots, x_{\sigma (i+m-1)}, x_{i+m}), \ldots, x_{m+n+1}),
\end{align*}
for $f \in C^{m+1} (\mathfrak{g}, \mathfrak{g})$, $g \in C^{n+1} (\mathfrak{g}, \mathfrak{g})$ and $x_1, \ldots, x_{m+n+1} \in \mathfrak{g}$. Then $(C^{\bullet +1 } (\mathfrak{g}, \mathfrak{g}), [~,~]_\mathsf{B})$ is a graded Lie algebra. The importance of this graded Lie algebra is given by the following result.

\begin{prop}
There is a one-to-one correspondence between Leibniz algebra structures on a vector space $\mathfrak{g}$ and Maurer-Cartan elements in the graded Lie algebra $(C^{\bullet +1 } (\mathfrak{g}, \mathfrak{g}), [~,~]_\mathsf{B})$.
\end{prop}

%\textcolor{red}{differential in terms of bracket}

\medskip

Next, we recall relative Rota-Baxter operators in the context of Leibniz algebras.

\begin{defn} (i) Let $\mathfrak{g}$ be a Leibniz algebra. A linear map $R : \mathfrak{g} \rightarrow \mathfrak{g}$ is said to be a {\bf Rota-Baxter operator} on $\mathfrak{g}$ if 
\begin{align}\label{rb-ident}
[R(x), R(y)]_\mathfrak{g} = R ([R(x), y]_\mathfrak{g} + [x, R(y)]_\mathfrak{g}), \text{ for } x, y \in \mathfrak{g}.
\end{align}

(ii) Let $\mathfrak{g}$ be a Leibniz algebra and $V$ be a representation of it. A linear map $R : V \rightarrow \mathfrak{g}$ is said to be a {\bf relative Rota-Baxter operator} (on $V$ over the Leibniz algebra $\mathfrak{g}$) if
\begin{align}
[R(v), R(v')]_\mathfrak{g} = R (l_V (R(v), v') + r_V (v, R(v') ) ), \text{ for } v, v' \in V.
\end{align}
It follows that a Rota-Baxter operator on a Leibniz algebra $\mathfrak{g}$ is a relative Rota-Baxter operator on the adjoint representation space $\mathfrak{g}$.
\end{defn}

%\textcolor{red}{write MC characterizations/ cohomology of relative RB operators depending on requirements}

\medskip

Next, we collect some necessary background on $L_\infty$-algebras due to Lada and Stasheff \cite{ls}. For our purpose, we follow the equivalent definition by a degree shift \cite{voro}. 

\begin{defn}\label{l-inf-def}
An $L_\infty$-algebra is a pair $(\mathfrak{L}, \{ l_k \}_{k \geq 1})$ consisting of a graded vector space $\mathfrak{L} = \oplus_{i \in \mathbb{Z}} \mathfrak{L}_i$ with a collection $\{ l_k \}_{k \geq 1}$ of degree $1$ linear maps (called the structure maps) $l_k : \mathfrak{L}^{\otimes k} \rightarrow \mathfrak{L}$, for $k \geq 1$, satisfying
\begin{itemize}
\item (graded symmetry):~ $l_k (x_{\sigma (1)}, \ldots, x_{\sigma (k)}) = \epsilon (\sigma) l_k (x_1, \ldots, x_k),$ for $k \geq 1$ and $\sigma \in S_k$,
\item (shifted higher Jacobi identity): for any $n \geq 1$ and homogeneous elements $x_1, \ldots, x_n \in \mathfrak{L}$,
\begin{align*}
\sum_{i+j = n+1} \sum_{\sigma \in Sh (i, n-i)} \epsilon (\sigma)~ l_j \big(   l_i (x_{\sigma (1)}, \ldots, x_{\sigma (i)}), x_{\sigma (i+1)}, \ldots, x_{\sigma (n)} \big) = 0.
\end{align*}
\end{itemize}
\end{defn}

%\begin{remark}
%It is important to mention that the $L_\infty$-algebra considered in the above definition is a degree $1$ shifted version of the $L_\infty$-algebra defined in \cite{lada-markl}. This is why the structure defined in Definition \ref{l-inf-def} is sometimes called an $L_\infty [1]$-algebra. However, in this paper, we use to call them $L_\infty$-algebras.
%\end{remark}

\begin{defn}\label{defn-mc}\cite{getzler,laza-rota} (i) An $L_\infty$-algebra $(\mathfrak{L}, \{ l_k \}_{k \geq 1})$ is said to be weakly filtered if there exists a descending filtration $\mathcal{F}_\bullet \mathfrak{L}$ of $\mathfrak{L}$ (that is, $\mathfrak{L} = \mathcal{F}_1 \mathfrak{L} \supset \cdots \supset \mathcal{F}_k \mathfrak{L} \supset \cdots$) and a natural number $N$ (called an index) such that $\mathfrak{L} \cong \lim_{n \rightarrow \infty} \mathfrak{L} / \mathcal{F}_n \mathfrak{L}$ and $l_k (\mathfrak{L}, \ldots, \mathfrak{L}) \subset \mathcal{F}_k \mathfrak{L}$, for all $k \geq N$.

\medskip

(ii) Let $(\mathfrak{L}, \{ l_k \}_{k \geq 1})$ be a weakly filtered $L_\infty$-algebra. An element $\theta \in \mathfrak{L}_0$ is said to be a Maurer-Cartan element if $\theta $ satisfies
\begin{align*}
\sum_{k =1}^\infty \frac{1}{k!} l_k (\theta, \ldots, \theta) = 0.
\end{align*}
\end{defn}

%One may twist a weakly filtered $L_\infty$-algebra by a Maurer-Cartan element \cite{getzler,laza-rota}.

%\begin{prop}\label{tw-l-inf}
%Let $(\mathfrak{L}, \{ l_k \}_{k \geq 1})$ be a weakly filtered $L_\infty$-algebra with filtration $\mathcal{F}_\bullet \mathfrak{L}$ and an index $N$. Suppose $\alpha \in \mathfrak{L}_0$ is a Maurer-Cartan element. Then $(\mathfrak{L}, \{ l^\alpha_k \}_{k \geq 1})$ is weakly filtered $L_\infty$-algebra with structure maps
%\begin{align*}
%l_k^\alpha (x_1, \ldots, x_k) := \sum_{i =0}^\infty l_{k+i} (\underbrace{\alpha, \ldots, \alpha}_{i}, x_1, \ldots, x_k),~ \text{ for } k \geq 1,
%\end{align*}
%and the same filtration $\mathcal{F}_\bullet \mathfrak{L}$ and a same index $N$. For any $\alpha' \in \mathfrak{L}_0$, the sum $\alpha + \alpha'$ is a Maurer-Cartan element in the $L_\infty$-algebra $(\mathfrak{L}, \{ l_k \}_{k \geq 1})$ if and only if $\alpha'$ is a Maurer-Cartan element in the $L_\infty$-algebra $(\mathfrak{L}, \{ l_k^\alpha \}_{k \geq 1}).$
%\end{prop}

\begin{defn}
A $V$-data is a quadruple $(L, \mathfrak{a}, P, \triangle)$ consisting of a graded Lie algebra $L$ (with the bracket $[~,~]$), an abelian graded Lie subalgebra $\mathfrak{a} \subset L$, a projection map $P : L \rightarrow L$ with $\mathrm{im} (P) = \mathfrak{a}$ and $\mathrm{ker}(P) \subset L$ a graded Lie subalgebra, and $\triangle \in \mathrm{ker}(P)_1$ satisfying $[\triangle, \triangle] = 0$.
\end{defn}

A $V$-data induce some $L_\infty$-algebras via Voronov's higher derived brackets \cite{voro}. We will use the following result in the next section to construct an $L_\infty$-algebra that characterizes relative Rota-Baxter Leibniz algebras as Maurer-Cartan elements.

\begin{thm}\label{v-thm}
Let $(L, \mathfrak{a}, P, \triangle)$ be a $V$-data. Then we have the following.

(i) The graded vector space $\mathfrak{a}$ can be equipped an $L_\infty$-algebra with the structure maps
\begin{align*}
l_k (a_1, \ldots, a_k ) = P [ \cdots [[\triangle, a_1], a_2], \ldots, a_k ], ~ \text{ for } k \geq 1.
\end{align*}

\medskip

(ii) If $L' \subset L$ is a graded Lie subalgebra so that $[\triangle, L'] \subset L'$, then the graded vector space $L'[1] \oplus \mathfrak{a}$ carries an $L_\infty$-algebra structure, where $(L'[1] \oplus \mathfrak{a})_i = (L'[1])_i \oplus \mathfrak{a}_i = L'_{i+1} \oplus \mathfrak{a}_i$ for all $i \in \mathbb{Z}$, and the structure maps are given by
\begin{align*}
l_1 (x[1]) =~& - [\triangle, x][1] +  P(x),\\
l_1 (a) =~& P [\triangle, a], \\
l_2 (x[1], y[1]) =~& (-1)^{|x|} [x, y][1],\\
l_k (x[1], a_1, \ldots, a_{k-1}) =~& P [\cdots [[x, a_1], a_2], \ldots, a_{k-1} ],~ k \geq 2,\\
l_k (a_1, \ldots, a_{k} ) =~& P [\cdots [[ \triangle, a_1], a_2], \ldots, a_k], ~ k \geq 2,
\end{align*}
for homogeneous elements $x, y \in L'$ (considered as elements $x[1], y[1] \in L'[1]$ with a degree shift) and homogeneous $a, a_1, \ldots, a_k \in \mathfrak{a}$. Up to the permutations of the above entries, all the other linear maps vanish. Moreover, $\mathfrak{a} \subset L'[1] \oplus \mathfrak{a}$ is an $L_\infty$-subalgebra.
\end{thm}

\section{rRB Leibniz algebras and their Maurer-Cartan characterizations}\label{sec-3}
In this section, we first define relative Rota-Baxter Leibniz algebras (rRB Leibniz algebras) and then construct an $L_\infty$-algebra whose Maurer-Cartan elements are rRB Leibniz algebras. We will use this characterization in Section \ref{sec-5} to define the cohomology of rRB Leibniz algebras.

\begin{defn}\label{rrbl} A rRB Leibniz algebra is a triple $(\mathfrak{g}, V, R)$ consisting of a Leibniz algebra $\mathfrak{g}$, a representation $V$ of the Leibniz algebra $\mathfrak{g}$ and a relative Rota-Baxter operator $R : V \rightarrow \mathfrak{g}$.
\end{defn}

\begin{notation}
For notational convenience, we denote a rRB Leibniz algebra by $V \xrightarrow{ R} \mathfrak{g}$ instead of the triple $(\mathfrak{g}, V, R)$. However, both of them capture the same informations.
\end{notation}

\begin{remark}\label{rmk-rb}
A Rota-Baxter Leibniz algebra (RB Leibniz algebra) is a pair $(\mathfrak{g}, R)$ that consists of a Leibniz algebra $\mathfrak{g}$ and a Rota-Baxter operator $R$ on it. Thus, a RB Leibniz algebra $(\mathfrak{g},R)$ can be considered as a rRB Leibniz algebra $\mathfrak{g} \xrightarrow{R} \mathfrak{g}$, where the domain of $R$ is the adjoint representation space $\mathfrak{g}$.
\end{remark}

\begin{defn}\label{rrbl-mor}
Let $V \xrightarrow{R} \mathfrak{g}$ and $V' \xrightarrow{R'} \mathfrak{g}'$ be two rRB Leibniz algebras. A morphism between them is a pair $(\phi, \psi)$ of a Leibniz algebra morphism $\phi : \mathfrak{g} \rightarrow \mathfrak{g}'$ and a linear map $\psi : V \rightarrow V'$ satisfying the following
\begin{align*}
\psi (l_V (x, v)) = l_{V'}^{\mathfrak{g}'} (\phi (x), \psi (v)), \quad \psi (r_V (v, x)) = r_{V'}^{\mathfrak{g}'} (\psi (v), \phi (x)) ~~~~ \text{ and } ~~~~ \phi \circ R = R' \circ \psi, ~ \text{ for } x \in \mathfrak{g}, v \in V.
\end{align*}
Here $l_{V'}^{\mathfrak{g}'}$ and $r_{V'}^{\mathfrak{g}'}$ denote the left and right $\mathfrak{g}'$-actions on $V'$, respectively. We write a morphism as above simply by the notation $(\phi, \psi) : (V \xrightarrow{R} \mathfrak{g}) \rightsquigarrow (V' \xrightarrow{R'} \mathfrak{g}').$
\end{defn}

\begin{prop}
Let $V \xrightarrow{R} \mathfrak{g}$ be a rRB Leibniz algebra. Then the vector space $V$ carries a Leibniz algebra structure with the bracket given by
\begin{align*}
[v, v']_R := l_V (R(v), v') + r_V (v, R(v')), ~ \text{ for } v, v' \in V.
\end{align*}
\end{prop}

We denote this Leibniz algebra simply by $V_R$, and called the Leibniz algebra induced by the rRB Leibniz algebra $V \xrightarrow{R} \mathfrak{g}$.
\medskip

\noindent {\bf Maurer-Cartan characterizations of rRB Leibniz algebras.} Here we construct an $L_\infty$-algebra that characterizes rRB Leibniz algebras as Maurer-Cartan elements. We first recall some notations from \cite{sheng-tang,DasSK}.

Let $\mathfrak{g}$ and $V$ be two vector spaces. For any linear map $f : \mathfrak{g}_{i_1} \otimes \cdots \otimes \mathfrak{g}_{i_n} \rightarrow \mathfrak{g}_j$ where $\mathfrak{g}_{i_1}, \ldots, \mathfrak{g}_{i_n} , \mathfrak{g}_j$ are either $\mathfrak{g}$ or $V$, we define a new linear map (called the horizontal lift) $\widehat{f} \in \mathrm{Hom} ((\mathfrak{g} \oplus V)^{\otimes n}, \mathfrak{g} \oplus V)$ by
\begin{align}\label{h-lift}
\widehat{f} := \begin{cases}
f  & \text{ on } \mathfrak{g}_{i_1} \otimes \cdots \otimes \mathfrak{g}_{i_n}\\
0 & \text{ on other entries.}
\end{cases}
\end{align}
Let $\mathfrak{g}^{k,l}$ be the direct sum of all possible $(k+l)$-tensor powers of $\mathfrak{g}$ and $V$ in which $\mathfrak{g}$ appears $k$-times (hence $V$ appears $l$ times). For instance,
\begin{align*}
\mathfrak{g}^{2,0} = \mathfrak{g} \otimes \mathfrak{g}, \quad \mathfrak{g}^{1,1} = (\mathfrak{g} \otimes V) \oplus (V \otimes \mathfrak{g}) ~~~ \text{ and } ~~~ \mathfrak{g}^{0,2} = V \otimes V.
\end{align*}
With this notation, we have $(\mathfrak{g} \oplus V)^{\otimes n} \cong \oplus_{k+l = n} \mathfrak{g}^{k,l}$. Moreover, there is an isomorphism (by the horizontal lift) 
\begin{align*}
C^n ( \mathfrak{g} \oplus V, \mathfrak{g} \oplus V) = \mathrm{Hom}((\mathfrak{g} \oplus V)^{\otimes n}, \mathfrak{g} \oplus V) \cong \big( \oplus_{k+l = n} \mathrm{Hom}(\mathfrak{g}^{k,l}, \mathfrak{g}) \big) \oplus \big( \oplus_{k+l = n} \mathrm{Hom}(\mathfrak{g}^{k,l}, V) \big).
\end{align*}

\begin{defn}
A linear map $f \in \mathrm{Hom} ((\mathfrak{g} \oplus V)^{\otimes n + 1}, \mathfrak{g} \oplus V)$ is said to be homogeneous of bidegree $k | l$ with $-1 \leq k, l \leq n+1$ and $k+l = n$ if
\begin{align*}
f (\mathfrak{g}^{k+1, l}) \subset \mathfrak{g}, \quad f (\mathfrak{g}^{k, l+1}) \subset V ~~~ \text{ and } ~~~ f (X) = 0 \text{ in all other cases}.
\end{align*}
We denote the set of all homogeneous linear maps of bidegree $k|l$ by $C^{k|l} (\mathfrak{g} \oplus V, \mathfrak{g} \oplus V)$.
\end{defn}

\begin{remark}
It is easy to see that there are isomorphisms (again by the horizontal lift)
\begin{align*}
C^{k | 0} (\mathfrak{g} \oplus V, \mathfrak{g} \oplus V) \cong \mathrm{Hom} (\mathfrak{g}^{\otimes k+1}, \mathfrak{g}) \oplus \mathrm{Hom} (\mathfrak{g}^{k,1}, V) ~~~ \text{ ~~ and ~~ } ~~~
C^{-1|l} (\mathfrak{g} \oplus V, \mathfrak{g} \oplus V) \cong \mathrm{Hom} (V^{\otimes l}, \mathfrak{g}).
\end{align*}
\end{remark}

Consider the graded Lie algebra $L = (C^{\bullet +1} (\mathfrak{g} \oplus V, \mathfrak{g} \oplus V), [~,~]_\mathsf{B})$ on the space of linear maps on $\mathfrak{g} \oplus V$ with the Balavoine bracket. The next result describes the behaviour of the Balavoine bracket on homogeneous linear maps \cite[Lemma 3.6]{sheng-tang}.

\begin{prop}  For $f \in C^{k_f | l_f}  (\mathfrak{g} \oplus V, \mathfrak{g} \oplus V)$ and $g \in C^{k_g | l_g} (\mathfrak{g} \oplus V, \mathfrak{g} \oplus V)$, we have
\begin{align*}
[f,g]_\mathsf{B} \in C^{k_f + k_g | l_f + l_g}  (\mathfrak{g} \oplus V, \mathfrak{g} \oplus V).
\end{align*}
\end{prop}

As a consequence of the above result, we get the following.

\begin{prop}
With the above notations, 

\begin{itemize}
\item[(i)] $L' = C^{\bullet | 0} (\mathfrak{g} \oplus V, \mathfrak{g} \oplus V) \subset L$ is a Lie subalgebra;

\item[(ii)] $\mathfrak{a} = C^{-1| \bullet +1} (\mathfrak{g} \oplus V, \mathfrak{g} \oplus V) \subset L$ is an abelian subalgebra.
\end{itemize}
\end{prop}

In the following result, we characterize Leibniz algebra structures on a vector space $\mathfrak{g}$ and representations on a vector space $V$ as Maurer-Cartan elements in the graded Lie algebra $L' = C^{\bullet | 0} (\mathfrak{g} \oplus V, \mathfrak{g} \oplus V)$. Recall that an element $\theta \in L_1$ in a graded Lie algebra $(L = \oplus_{i \in \mathbb{Z}} L_i, [~,~])$ is said to be a Maurer-Cartan element if $[\theta, \theta] = 0$. This notion is consistent with Maurer-Cartan elements in an $L_\infty$-algebra (Definition \ref{defn-mc} (ii)). Let $(L = \oplus_{i \in \mathbb{Z}} L_i, [~,~])$ be a graded Lie algebra. Then it can be checked that $(L[1], \{ l_k \}_{k=1}^\infty)$ is an $L_\infty$-algebra, where $l_2 (x[1], y[1]) = (-1)^{|x|} [x, y][1]$, and $l_k = 0$ for $k \neq 2$. Then an element $\theta \in L_1$ is a Maurer-Cartan element in the graded Lie algebra $L$ if and only if the element $\theta [1]$ is a Maurer-Cartan element in the $L_\infty$-algebra $L[1]$.

\begin{thm}\label{thm-lr-both}
Let $\mathfrak{g}$ and $V$ be two vector spaces. Let $\mu_\mathfrak{g} \in \mathrm{Hom} (\mathfrak{g}^{\otimes 2}, \mathfrak{g})$, $l_V \in \mathrm{Hom} (\mathfrak{g} \otimes V, V)$ and $r_V \in \mathrm{Hom}(V \otimes \mathfrak{g}, V)$ be linear maps. Then $\mathfrak{g} = (\mathfrak{g}, \mu_\mathfrak{g})$ is a Leibniz algebra and $V = (V, l_V, r_V)$ is a representation of the Leibniz algebra $\mathfrak{g}$ if and only if $\pi = \mu_\mathfrak{g} + l_V + r_V \in (L')_1 = C^{1|0} (\mathfrak{g} \oplus V, \mathfrak{g} \oplus V)$ is a Maurer-Cartan element in the graded Lie algebra $L'$.
\end{thm}

\begin{proof}
For $(x,v), (y, v'), (z, v'') \in \mathfrak{g} \oplus V$, we have
\begin{align}
&[\pi, \pi]_\mathsf{B} \big( (x,v), (y, v'), (z, v'') \big) \nonumber \\
&= 2 \big\{   \pi \big(   \pi ((x,v), (y, v')) , (z, v'')   \big) - \pi \big( (x,v),  \pi ((y,v'), (z, v''))  \big)  + \pi \big( (y, v'), \pi ((x,v), (z, v''))  \big) \big\} \nonumber \\
&= 2 \big(  [[x,y]_\mathfrak{g}, z ]_\mathfrak{g} - [x, [y, z]_\mathfrak{g} ]_\mathfrak{g} + [y, [x, z]_\mathfrak{g} ]_\mathfrak{g}, ~ l_V ([x, y]_\mathfrak{g}, v'') - l_V (x, l_V (y, v'')) + l_V (y, l_V (x, v'')) \label{mc-leibrep}\\
& ~~~~ + r_V (l_V(x, v'), z) - l_V (x, r_V (v', z)) + r_V (v', [x,z]_\mathfrak{g}) + r_V (r_V (v, y), z) - r_V (v, [y, z]_\mathfrak{g}) + l_V (y, r_V (v, z))  \big). \nonumber
\end{align}
Here we use the notation $[x,y]_\mathfrak{g} = \mu_\mathfrak{g} (x,y)$, for $x,y \in \mathfrak{g}$. Note that the expression (\ref{mc-leibrep}) implies that $[\pi, \pi]_\mathsf{B} = 0$ if and only if $(\mathfrak{g}, \mu_\mathfrak{g})$ is a Leibniz algebra and $V = (V, l_V, r_V)$ is a representation of it.
\end{proof}

Let $\mathfrak{g}$ and $V$ be two vector spaces. Then the quadruple
\begin{align*}
\big(   L = (C^{\bullet +1} (\mathfrak{g} \oplus V, \mathfrak{g} \oplus V), [~,~]_\mathsf{B}),~ \mathfrak{a} = C^{-1| \bullet +1} (\mathfrak{g} \oplus V, \mathfrak{g} \oplus V),~ P,~ \triangle = 0 \big)
\end{align*}
is a $V$-data, where $P : L \rightarrow L$ is the projection onto the subspace $\mathfrak{a}$. Moreover, $L' = C^{\bullet | 0} (\mathfrak{g} \oplus V, \mathfrak{g} \oplus V) \subset L$ is a Lie subalgebra. Hence by applying Theorem \ref{v-thm} (ii), we get the following result.

\begin{thm}
There is an $L_\infty$-algebra structure on the graded vector space $L' [1] \oplus \mathfrak{a}$ with structure maps
\begin{align*}
l_2 (q[1], q'[1]) =~& (-1)^{|q|} [q, q']_\mathsf{B} [1],\\
l_k (q[1], a_1, \ldots, a_{k-1}) =~& P [ \cdots [[ q, a_1]_\mathsf{B}, a_2 ]_\mathsf{B}, \ldots, a_{k-1} ]_\mathsf{B},
\end{align*}
for homogeneous elements $q, q' \in L'$ and $ a_1, \ldots, a_{k-1} \in \mathfrak{a}$.
\end{thm}

\begin{remark}
The $L_\infty$-algebra $(L'[1] \oplus \mathfrak{a}, \{ l_k \}_{k \geq 1})$ constructed in the above theorem is weakly filtered with a filtration 
\begin{align*}
\mathcal{F}_1 = L'[1] \oplus \mathfrak{a}, \quad \mathcal{F}_2 = P [ L'[1] \oplus \mathfrak{a}, \mathfrak{a} ]_\mathsf{B} ~~ \text{ and } ~~ \mathcal{F}_k = P [ \cdots [ L'[1] \oplus \mathfrak{a}, \mathfrak{a}]_\mathsf{B}, \ldots, \mathfrak{a} ]_\mathsf{B}, ~\text{for } k \geq 3,
\end{align*}
and an index $N = 3$.
\end{remark}

We observe that the degree $0$ component of the graded vector space $L'[1] \oplus \mathfrak{a}$ is given by
\begin{align*}
(L'[1] \oplus \mathfrak{a})_0 = (L')_1 \oplus \mathfrak{a}_0 = \underbrace{  \mathrm{Hom} (\mathfrak{g}^{\otimes 2}, \mathfrak{g}) \oplus \mathrm{Hom}(\mathfrak{g}^{1,1}, V)  }_{ (L')_1} \oplus \underbrace{\mathrm{Hom} (V, \mathfrak{g})}_{\mathfrak{a}_0}.
\end{align*}

Suppose there are maps $\mu_\mathfrak{g} \in \mathrm{Hom} (\mathfrak{g}^{\otimes 2}, \mathfrak{g})$, $l_V \in \mathrm{Hom}(\mathfrak{g} \otimes V , V)$, $r_V \in \mathrm{Hom} (V \otimes \mathfrak{g}, V)$ and $R \in \mathrm{Hom}(V, \mathfrak{g})$. Note that  $\pi = \mu_\mathfrak{g} + l_V + r_V \in (L')_1$ can be considered as an element $\pi[1] \in (L'[1])_0$. We take the element
$\theta = (\pi [1], R) \in (L'[1] \oplus \mathfrak{a})_0$.

\begin{thm}\label{l-inf-mc}
With the above notations, $V \xrightarrow{R} \mathfrak{g}$ is a rRB Leibniz algebra if and only if $\theta = (\pi[1], R) \in (L'[1] \oplus \mathfrak{a})_0$ is a Maurer-Cartan element in the $L_\infty$-algebra $(L'[1] \oplus \mathfrak{a}, \{ l_k \}_{k \geq 1})$.
\end{thm}

\begin{proof}
First observe that $l_1 \big((\pi [1], R) \big) = 0$ (from the definition of $l_1$). Moreover, for homogeneous bidegree reason, we have $[[[ \pi, R ]_\mathsf{B}, R]_\mathsf{B}, R]_\mathsf{B} = 0$. This implies that $l_k \big( (\pi[1], R), \ldots, (\pi[1], R)   \big) = 0$, for $k \geq 4.$ Hence
\begin{align}
&\sum_{k=1}^\infty \frac{1}{k!} \big( (\pi[1], R), \ldots, (\pi[1], R) \big) \nonumber \\
&= \frac{1}{2!} l_2 \big( (\pi[1], R), (\pi[1], R) \big) + \frac{1}{3!} l_3 \big( (\pi[1], R), (\pi[1], R), (\pi[1], R) \big) \nonumber \\
&= \big(  - \frac{1}{2} [\pi, \pi]_\mathsf{B} [1], ~ \frac{1}{2} [[ \pi, R]_\mathsf{B}, R ]_\mathsf{B}  \big). \label{mc-pi1}
\end{align}
Note that
\begin{align*}
[\pi, \pi]_\mathsf{B} = 0 &\leftrightsquigarrow (\mathfrak{g}, \mu_\mathfrak{g}) \text{ is a Leibniz algebra and } (V, l_V, r_V) \text{ is representation (cf. Theorem \ref{thm-lr-both})},\\
[[ \pi, R]_\mathsf{B}, R ]_\mathsf{B} = 0 &\leftrightsquigarrow R : V \rightarrow \mathfrak{g} \text{ is a relative Rota-Baxter operator (see \cite{sheng-tang})}.
\end{align*}
Therefore, the expression in (\ref{mc-pi1}) vanishes if and only if $V  \xrightarrow{R} \mathfrak{g}$ is a rRB Leibniz algebra. This completes the proof.
\end{proof}

Given an $L_\infty$-algebra and a Maurer-Cartan element, one can construct a new $L_\infty$-algebra \cite{getzler,laza-rota}. The following result generalizes this in the present context.

\begin{thm}\label{twisted-l-inf}
Let $V \xrightarrow{R} \mathfrak{g}$ be a rRB Leibniz algebra with the corresponding Maurer-Cartan element $\theta = (\pi [1], R) \in (L'[1] \oplus \mathfrak{a})_0$ in the $L_\infty$-algebra $(L'[1] \oplus \mathfrak{a}, \{ l_k \}_{k \geq 1})$. Then we have the following.

(i) There is an $L_\infty$-algebra $(L'[1] \oplus \mathfrak{a}, \{ l_k^{(\pi[1],R)} \}_{k \geq 1})$, where
\begin{align*}
l_k^{(\pi [1], R)} \big( (x_1[1], a_1), \ldots, (x_k [1], a_k)  \big) = \sum_{i=0}^\infty ~l_{k+i} \big( \underbrace{(\pi[1], R), \ldots, (\pi[1], R)}_{i}, (x_1[1], a_1), \ldots, (x_k [1], a_k) \big), \text{ for } k \geq 1.
\end{align*}

(ii) For any $\theta' = (\pi'[1], R') \in (L'[1] \oplus \mathfrak{a})_0$, the sum $\theta + \theta' = ((\pi + \pi' )[1], R+ R')$ defines a rRB Leibniz algebra if and only if $\theta' = (\pi'[1], R')$ is a Maurer-Cartan element in the $L_\infty$-algebra $(L'[1] \oplus \mathfrak{a}, \{ l_k^{(\pi[1],R)} \}_{k \geq 1})$.
\end{thm}

\section{Representations of rRB Leibniz algebras}\label{sec-4}
In this section, we define representations of an rRB Leibniz algebra and provide some examples. Given an rRB Leibniz algebra and a representation, we also construct the semidirect product rRB Leibniz algebra.

\begin{defn}\label{repn-defn}
Let $V \xrightarrow{R} \mathfrak{g}$ be a rRB Leibniz algebra. A {\bf representation} of it consists of a triple $(W \xrightarrow{S} \mathfrak{h}, l, r)$ in which $W \xrightarrow{S} \mathfrak{h}$ is a $2$-term chain complex with both $\mathfrak{h}$ and $W$ are representations of the Leibniz algebra $\mathfrak{g}$, and there are linear maps (called the left and right pairings) $l : V \otimes \mathfrak{h} \rightarrow W$ and $r : \mathfrak{h} \otimes V \rightarrow W$ satisfying the following set of identities
\begin{align}\label{rep-1s}
\begin{cases} l_W (x, l(v, h)) = l (l_V (x,v), h) + l (v, l_\mathfrak{h} (x, h)),\\
l (v, l_\mathfrak{h} (x, h)) = l (r_V (v, x), h) + l_W (x, l (v, h)),\\
l (v, r_\mathfrak{h} (h, x)) = r_W (l (v, h), x) + r (h, r_V (v, x)),
\end{cases}
\end{align}
\begin{align}\label{rep-2s}
\begin{cases}  l_W (x, r(h, v)) = r (l_\mathfrak{h} (x, h), v) + r (h, l_V (x, v)), \\
r (h, l_V (x, v)) = r (r_\mathfrak{h} (h, x), v) + l_W (x, r (h, v)),\\
r (h, r_V (v, x)) = r_W (r (h, v), x) + l ( v, r_\mathfrak{h}(h, x)),
\end{cases}
\end{align}
\begin{align}\label{rep-3s}
\begin{cases}
l_\mathfrak{h} (R(v), S(w)) = S \big( l_W (R(v), w) + l (v, S(w)) \big),\\
r_\mathfrak{h} (S(w), R(v)) = S \big( r (S (w), v) + r_W (w, R(v)) \big),
\end{cases}
\end{align}
for $x \in \mathfrak{g}$, $v \in V$, $h \in \mathfrak{h}$ and $w \in W$.
\end{defn}

Note that all the identities in (\ref{rep-1s}) and (\ref{rep-2s}) are analogs of the Leibniz identity (\ref{leib-iden}). Hence all of them can be understood by the following string diagram with suitable inputs and suitable maps on the boxes below:\\\\

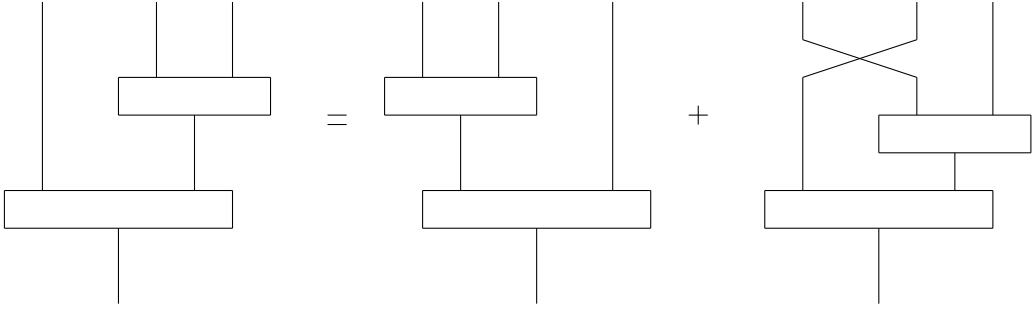
\begin{center}
\begin{tikzpicture}[scale=0.5]
 \draw (3,0) -- (3,2);	 \draw (0,2) -- (6,2); \draw (0,2) -- (0,3); \draw (0,3) -- (6,3); \draw (6,2) -- (6,3); \draw (1,3) -- (1,8); \draw (3,6) -- (7,6); \draw (3,5) -- (3,6); \draw (3,5) -- (7,5); \draw (7,5) -- (7,6); \draw (5,3) -- (5,5); \draw (4,8) -- (4,6); \draw (6,8) -- (6,6); \draw (8.5,5) -- (9,5); \draw (8.5, 4.75) -- (9,4.75); \draw (10,5) -- (14,5); \draw (14,6) -- (10,6); \draw (10,5) -- (10,6); \draw (14,5) -- (14,6); \draw (11,6) -- (11,8); \draw (13,6) -- (13,8); \draw (12,3) -- (12,5); \draw (11,3) -- (17,3); \draw (11,2) -- (17,2); \draw (11,2) -- (11,3); \draw (17,2) -- (17,3); \draw (16,3) -- (16,8); \draw (14,0) -- (14,2); \draw (18,5) -- (18.5,5); \draw (18.25,4.75) -- (18.25,5.25); \draw (23,0) -- (23,2); \draw (20,2) -- (26,2); \draw (20,3) -- (26,3); \draw (20,3) -- (20,2); \draw (26,3) -- (26,2); \draw (21,6) -- (21,3); \draw (25,4) -- (25,3); \draw (23,4) -- (27,4); \draw (23,5) -- (27,5); \draw (23,5) -- (23,4); \draw (27,5) -- (27,4); \draw (24,6) -- (24,5); \draw (26,8) -- (26,5);\draw (21,8) -- (21,7); \draw (24,8) -- (24,7); \draw (21,6) -- (24,7); \draw (24,6) -- (21,7); 
\end{tikzpicture}
\end{center}

\medskip

\noindent On the other hand, both the identities in (\ref{rep-3s}) are analogs of the Rota-Baxter identity (\ref{rb-ident}). Thus, both of them can be understood using suitable string diagram. By viewing rRB Leib algebras and their representations using string diagrams, one can easily generalize these to symmetric monoidal categories.

\begin{exam} (Adjoint representation) Let $V \xrightarrow{R} \mathfrak{g}$ be a rRB Leibniz algebra. Then the triple $(V \xrightarrow{R} \mathfrak{g}, l_\mathrm{ad}, r_\mathrm{ad})$ is a representation of the rRB Leibniz algebra $V \xrightarrow{R} \mathfrak{g}$, where $\mathfrak{g}$ is equipped with the adjoint representation and $V$ is equipped with the given representation of the Leibniz algebra $\mathfrak{g}$. Moreover, the pairings are given by $l_\mathrm{ad}= r_V : V \otimes \mathfrak{g} \rightarrow V$ and $r_\mathrm{ad}= l_V : \mathfrak{g} \otimes V \rightarrow V$. This is called the adjoint representation.
\end{exam}

\begin{exam}\label{rpr-rb} (Representations of RB Leibniz algebras)
Let $(\mathfrak{g}, R)$ be a RB Leibniz algebra. A representation of $(\mathfrak{g}, R)$ consists of a pair $(V, R_V)$ where $V$ is a representation of $\mathfrak{g}$ and $R_V: V \rightarrow V$ is a linear map satisfying for $x \in \mathfrak{g}$, $v \in V$,
\begin{align*}
l_V (R(x), R_V(v)) =~& R_V \big(  l_V (R(x), v) + l_V (x, R_V (v))  \big),\\
r_V ( R_V(v), R(x)) =~& R_V \big( r_V ( R_V (v), x) + r_V (v, R(x)) \big).
\end{align*}
Then it is easy to see that the triple $(V \xrightarrow{R_V} V, l_V, r_V)$ is a representation of the rRB Leibniz algebra $\mathfrak{g} \xrightarrow{R} \mathfrak{g}.$
\end{exam}

\begin{exam}
Let $V \xrightarrow{R} \mathfrak{g}$  and $V' \xrightarrow{R'} \mathfrak{g}'$ be two rRB Leibniz algebras and $(\phi, \psi)$ be a morphism between them. Then the triple $(V' \xrightarrow{R'} \mathfrak{g}', l, r)$ is a representation of the rRB Leibniz algebra $V \xrightarrow{R} \mathfrak{g}$, where $\mathfrak{g}'$ and $V'$ are representations of $\mathfrak{g}$ via
\begin{align*}
l_{\mathfrak{g}'} (x, x') := [\phi (x), x']_{\mathfrak{g}'}, ~~~~ r_{\mathfrak{g}'} (x', x) := [x', \phi (x)]_{\mathfrak{g}'} ~~~ \text{ ~~ and ~~ } ~~~ l_{V'} (x, v') := l_{V'}^{\mathfrak{g}'} (\phi (x), v'), ~~~~ r_V (v', x) := r_{V'}^{\mathfrak{g}'} (v', \phi (x)),
\end{align*} 
for $x \in \mathfrak{g}$, $x' \in \mathfrak{g}'$ and $v' \in V'$. Here $l_{V'}^{\mathfrak{g}'}$ and $r_{V'}^{\mathfrak{g}'}$ denote the left and right $\mathfrak{g}'$-actions on $V'$, respectively. Finally, the pairings $l : V \otimes \mathfrak{g}' \rightarrow V'$ and $r : \mathfrak{g}' \otimes V \rightarrow V'$ are given by
\begin{align*}
l (v,x') := r_{V'}^{\mathfrak{g}'} (\psi (v), x')  ~~~ \text{ ~~ and ~~ } ~~~ r (x', v) := l_{V'}^{\mathfrak{g}'} (x', \psi (v)), ~ \text{ for } v \in V,~ x' \in \mathfrak{g}'.
\end{align*}
\end{exam}

\begin{exam}
Let $V \xrightarrow{R} \mathfrak{g}$ be a rRB Leibniz algebra. Consider the Lie algebra $\mathfrak{g}_\mathrm{Lie}$ which is the quotient of $\mathfrak{g}$ by the ideal generated by the elements $[x,x]_\mathfrak{g}$, for $x \in \mathfrak{g}$. We also consider the space $V_\mathrm{Lie}$ which is the quotient of $V$ by the subspace generated by elements of the form $l_V (x,v) + r_V (v,x)$, for $x \in \mathfrak{g}$ and $v \in V$. Then $V_\mathrm{Lie}$ is a representation of the Lie algebra $\mathfrak{g}_\mathrm{Lie}$ with the action $\mathfrak{g}_\mathrm{Lie} \otimes V_\mathrm{Lie} \rightarrow V_\mathrm{Lie}$,  $( \lfloor x \rfloor , \lfloor v \rfloor) \mapsto \lfloor l_V (x, v) \rfloor = - \lfloor r_V (v, x) \rfloor$, for $\lfloor x \rfloor \in \mathfrak{g}_\mathrm{Lie}$ and $\lfloor v \rfloor \in V_\mathrm{Lie}$. Here $\lfloor ~ \rfloor$ denotes the class of an element. Hence $\mathfrak{g}_\mathrm{Lie}$ can be considered as a Leibniz algebra and $V_\mathrm{Lie}$ a representation of it.
Note that the map $R : V \rightarrow \mathfrak{g}$ induces a map $\lfloor R \rfloor : V_\mathrm{Lie} \rightarrow \mathfrak{g}_\mathrm{Lie}$ given by $\lfloor R \rfloor \big( \lfloor v \rfloor \big) = \lfloor R(v) \rfloor$, for $\lfloor v \rfloor \in V_\mathrm{Lie}$. In fact, $V_\mathrm{Lie} \xrightarrow{\lfloor R \rfloor} \mathfrak{g}_\mathrm{Lie}$ is a rRB Leibniz algebra. Then the given rRB Leibniz algebra $V \xrightarrow{R} \mathfrak{g}$ is a representation of the rRB Leibniz algebra $V_\mathrm{Lie} \xrightarrow{\lfloor R \rfloor} \mathfrak{g}_\mathrm{Lie}$, where $\mathfrak{g}$ and $V$ are representations of the Leibniz algebra $\mathfrak{g}_\mathrm{Lie}$ with left and right $\mathfrak{g}_\mathrm{Lie}$-actions
\begin{align*}
l^{\mathfrak{g}_\mathrm{Lie}}_\mathfrak{g} (\lfloor x \rfloor, y) = [x,y]_\mathfrak{g}, \quad r^{\mathfrak{g}_\mathrm{Lie}}_\mathfrak{g} (y, \lfloor x \rfloor) = [y,x]_\mathfrak{g}, \quad l^{\mathfrak{g}_\mathrm{Lie}}_V ( \lfloor x \rfloor, v) = l_V (x, v) ~~ \text{ and } ~~ r^{\mathfrak{g}_\mathrm{Lie}}_V (v, \lfloor x \rfloor) = r_V (v,x),
\end{align*}
for $\lfloor x \rfloor \in \mathfrak{g}_\mathrm{Lie}$, $y \in \mathfrak{g}$ and $v \in V$. The pairing maps $l : V_\mathrm{Lie} \otimes \mathfrak{g} \rightarrow V$ and $r: \mathfrak{g} \otimes V_\mathrm{Lie} \rightarrow V$ are respectively given by $l (\lfloor v \rfloor , x) = r_V (v, x)$ and $r (x, \lfloor v \rfloor) = l_V (x, v)$, for $\lfloor v \rfloor \in V_\mathrm{
Lie}$, $x \in \mathfrak{g}$.
\end{exam}

Let $\mathfrak{g}$ be a Leibniz algebra and $V$ be a representation. Then the dual vector space $V^*$ also carries a representation of the Leibniz algebra $\mathfrak{g}$ with left and right $\mathfrak{g}$-actions $l_{V^*} : \mathfrak{g} \otimes V^* \rightarrow V^*$ and $r_{V^*} : V^* \otimes \mathfrak{g} \rightarrow V^*$ are given by
\begin{align*}
l_{V^*} (x, f_V) (v)= - f_V (l_V (x,v)) ~~~ \text{ and } ~~~ r_{V^*} (f_V, x)(v) = f_V (l_V (x, v) + r_V (v,x)),
\end{align*}
for $x \in \mathfrak{g}$, $f_V \in V^*$ and $v \in V.$ See \cite[Lemma 2.10]{sheng-tang} for more details. With this notion of dual representations of a Leibniz algebra, we can dualize representations of a rRB Leibniz algebra.

\begin{prop}\label{prop-dual}
Let $V \xrightarrow{R} \mathfrak{g}$ be a rRB Leibniz algebra and $(W \xrightarrow{S} \mathfrak{h}, l, r)$ be a representation of it. Then $ (\mathfrak{h}^* \xrightarrow{-S^*} W^*, l^*, r^*)$ is also a representation, where $W^*, \mathfrak{h}^*$ are equipped with dual representations of the Leibniz algebra $\mathfrak{g}$, and the pairings $l^* : V \otimes W^* \rightarrow \mathfrak{h}^*$ and $r^* : W^* \otimes V \rightarrow \mathfrak{h}^*$ are given by
\begin{align*}
l^* (v, f_W) (h) = - f_W (l (v, h)) ~~~ \text{ and } ~~~ r^* (f_W, v)(h) = f_W ( l (v, h) + r (h, v)), ~ \text{ for } v \in V, f_W \in W^*, h \in \mathfrak{h}.
\end{align*}
\end{prop}

\begin{proof}
For any $x \in \mathfrak{g}$, $v \in V$, $f_W \in W^*$ and $h \in \mathfrak{h}$, we have
\begin{align*}
&l_{\mathfrak{h}^*} (x, l^* (v, f_W)) (h) = - l^* (v, f_W) (l_\mathfrak{h} (x, h)) = f_W \big(  l (v, l_\mathfrak{h} (x, h))  \big), \\
&l^* (l_V (x, v), f_W) (h) + l^* (v, l_{W^*} (x, f_W)) (h) = - f_W \big(  l (   l_\mathfrak{h}(x, v), h) \big) - l_{W^*} (x, f_W) (l(v, h)) \\
&\qquad \qquad \qquad \qquad \qquad \qquad \qquad \qquad \qquad = - f_W \big(  l (   l_\mathfrak{h}(x, v), h) \big) + f_W \big( l_W (x, l(v, h))  \big).
\end{align*}
It follows from (\ref{rep-1s}) that $l_{\mathfrak{h}^*} (x, l^* (v, f_W)) = l^* (l_V (x, v), f_W) + l^* (v, l_{W^*} (x, f_W))$. Similarly, 
\begin{align*}
&l^* (v, l_{W^*} (x, f_W)) (h) = - l_{W^*} (x, f_W) (l (v, h)) = f_W \big( l_W (x, l(v, h))  \big), \\
&l^* (r_V (v, x), f_W) (h) + l_{\mathfrak{h}^*} (x, l^*(x, f_W)) (h) = - f_W ( l(r_V (v, x), h)) - l^* (v, f_W) (l_\mathfrak{h} (x, h)) \\
&\qquad \qquad \qquad \qquad \qquad \qquad \qquad \qquad \qquad = - f_W \big( l (r_V (v, x), h)   \big) + f_W \big( l (v, l_\mathfrak{h} (x, h)) \big).
\end{align*}
Hence from (\ref{rep-1s}), we get $l^* (v, l_{W^*} (x, f_W)) = l^* (r_V (v, x), f_W) + l_{\mathfrak{h}^*} (x, l^*(x, f_W))$. We also have
\begin{align*}
&l^* (v, r_{W^*} (f_W, x)) (h) = - r_{W^*} (f_W, x) (l (v, h)) = - f_W \big(  l_W (x, l(v, h)) + r_W (l(v, h), x) \big),\\
&r_{\mathfrak{h}^*} (l^* (v, f_W), x) (h) + r^* (f_W, r_V (v, x)) (h) \\
&= l^* (v, f_W) \big(   l_\mathfrak{h} (x, h) + r_\mathfrak{h} (h, x) \big) + f_W \big(  l (r_V (v, x), h) \big) + f_W \big( r (h, r_V (v, x) )  \big) \\
&= - f_W \big(  l (v, l_\mathfrak{h} (x, h) )  \big) - f_W \big( l (v, r_\mathfrak{h}(h, x))  \big) + f_W \big(  l (r_V (v, x), h) \big) + f_W \big( r (h, r_V (v, x) )  \big).
\end{align*}
Thus, we have $l^* (v, r_{W^*} (f_W, x)) = r_{\mathfrak{h}^*} (l^* (v, f_W), x) + r^* (f_W, r_V (v, x))$  (follows from (\ref{rep-1s})). Therefore, the identities in (\ref{rep-1s}) hold for the dual structures. By similar observations, we can show that the identities in (\ref{rep-2s}) hold for dual structures. Finally, for $v \in V$, $f_\mathfrak{h} \in \mathfrak{h}^*$ and $w \in W$,
\begin{align*}
&l_{W^*} (R(v), -S^* (f_\mathfrak{h})) (w) + S^* \big( l_{\mathfrak{h}^*} (R(v), f_\mathfrak{h}) + l^* (v, -S^* (f_\mathfrak{h}))   \big) (w) \\
&= S^* (f_\mathfrak{h}) ( l_W (R(v), w)) + l_{\mathfrak{h}^*} (R(v), f_\mathfrak{h}) (S(w)) - l^* (v, S^* (f_\mathfrak{h})) (S(w)) \\
&= f_\mathfrak{h} \big( S (l_W (R(v), w ) )  \big) - f_\mathfrak{h} \big( l_\mathfrak{h} (R(v), S(w))  \big) + f_\mathfrak{h} \big(  S (l (v, S(w)))  \big) = 0 \quad (\text{by } (\ref{rep-3s}))
\end{align*}
and
\begin{align*}
&r_{W^*} ( -S^* (f_\mathfrak{h}), R(v))(w) + S^* \big( r^* (-S^* (f_\mathfrak{h}), v) + r_{\mathfrak{h}^*} (f_\mathfrak{h}, R(v)) \big)(w) \\
&= -S^* (f_\mathfrak{h}) (l_W (R(v), w) + r_W (w, R(v))) - r^* (S^* (f_\mathfrak{h}), v) (S (w)) + r_{\mathfrak{h}^*} (f_\mathfrak{h}, R (v)) (S (w)) \\
&= - f_\mathfrak{h} \big(  S (l_W (R(v), w)) + S (r_W (w, R(v)))   \big) - f_\mathfrak{h} \big(  S (l (v, S (w))) + S (r (S(w), v))  \big) \\
& \qquad \qquad \qquad \qquad + f_\mathfrak{h} (l_\mathfrak{h} (R(v), S (w))) + f_\mathfrak{h} (r_\mathfrak{h} (S (w), R(v))) = 0 \quad (\text{by } (\ref{rep-3s})).
\end{align*}
This verifies the identities in (\ref{rep-3s}) for dual structures. Hence $(\mathfrak{h}^* \xrightarrow{-S^*} W^*, l^*, r^*)$ is a representation.
\end{proof}

\begin{exam} (Coadjoint representation) Let $V \xrightarrow{R} \mathfrak{g}$ be a rRB Leibniz algebra. Then $(\mathfrak{g}^* \xrightarrow{-R^*} V^*, l^*_{\mathrm{ad}}, r^*_{\mathrm{ad}})$ is a representation of the rRB Leibniz algebra $V \xrightarrow{R} \mathfrak{g}$, where the pairings $l^*_{\mathrm{ad}} : V \otimes V^* \rightarrow \mathfrak{g}^*$ and $r^*_{\mathrm{ad}} : V^* \otimes V \rightarrow \mathfrak{g}^*$ are given by
\begin{align*}
l^*_{\mathrm{ad}} (v, f_V)(x) = - f_V (r_V (v,x)) ~~~ \text{ and } ~~~ r^*_{\mathrm{ad}} (f_V, v)(x) = f_V (l_V (x,v) + r_V (v,x)), ~ \text{ for } v \in V, f_V \in V^*, x \in \mathfrak{g}.
\end{align*}
This is called the coadjoint representation.
\end{exam}

In the following result, we show that a representation of an rRB Leibniz algebra gives rise to a representation of the induced Leibniz algebra.

\begin{prop}\label{vr-h}
Let $V \xrightarrow{R} \mathfrak{g}$ be a rRB Leibniz algebra and $(W \xrightarrow{S} \mathfrak{h}, l, r)$ be a representation of it. Then the vector space $\mathfrak{h}$ carries a representation of the induced Leibniz algebra $V_R$ with left and right $V_R$-actions $l_\triangleright : V_R \otimes \mathfrak{h} \rightarrow \mathfrak{h}$ and $r_\triangleleft : \mathfrak{h} \otimes V_R \rightarrow \mathfrak{h}$ are given by 
\begin{align*}
l_\triangleright (v, h) = l_\mathfrak{h} (R(v), h) - S \circ l (v, h) ~~~ \text{ and } ~~~ r_\triangleleft (h, v) = r_\mathfrak{h} (h, R(v)) - S \circ r (h, v), ~ \text{ for }  v \in V_R, h \in \mathfrak{h}.
\end{align*}
\end{prop}

\begin{proof}
For $v, v' \in V_R$ and $h \in \mathfrak{h}$, we have
\begin{align*}
&l_\triangleright (v, l_\triangleright (v', h)) - l_\triangleright ([v,v']_R, h) - l_\triangleright (v', l_\triangleright (v,h) ) \\
&= l_\mathfrak{h} (R(v) , l_\triangleright (v', h)) - S \circ l (v, l_\triangleright (v', h)) - l_\mathfrak{h} (R[v, v']_R, h) + S \circ l ([v,v']_R, h) \\
& \qquad \qquad- l_\mathfrak{h} (R(v'), l_\triangleright (v, h)) + S \circ l (v', l_\triangleright (v, h)) \\
&= \cancel{l_\mathfrak{h} (R(v), l_\mathfrak{h} (R(v'), h)  )} - l_\mathfrak{h} (R(v), S \circ l (v', h)) - S \circ l (v, l_\mathfrak{h} (R(v'), h)) + S \circ l (v, S \circ l (v', h)) \\
&\qquad \qquad - \cancel{l_\mathfrak{h} ([R(v), R(v')]_\mathfrak{g}, h)} + S \circ l (l_V (R(v), v'), h) + S \circ l (r_V (v, R(v')), h) \\
&\qquad \qquad - \cancel{l_\mathfrak{h} (R(v'), l_\mathfrak{h} (R(v), h))} + l_\mathfrak{h} (R(v'), S \circ l (v, h)) + S \circ l (v', l_\mathfrak{h} (R(v), h)) - S \circ l (v', S \circ l (v, h)) \\
&= - S \big( l_W (R(v), l(v', h)) + l (v, S \circ l (v', h))   \big) - S \circ l (v, l_\mathfrak{h} (R(v'), h)) + S \circ l (v, S \circ l (v', h)) \\
& \qquad \qquad + S \circ l (l_V (R(v), v'), h) + S \circ l (r_V (v, R(v')) , h) + S \big(  l_W (R(v'), l(v, h)) + l (v', S \circ l (v,h))  \big) \\ 
& \qquad \qquad +S \circ l (v', l_\mathfrak{h} (R(v), h)) - S \circ l (v', S \circ l (v, h)) \\
&= 0.
\end{align*}
Similarly, it is straightforward to verify that
\begin{align*}
l_\triangleright (v, r_\triangleleft (h, v')) - r_\triangleleft (l_\triangleright (v,h), h') - r_\triangleleft (h, [v, v']_R) = 0,\\
r_\triangleleft (h, [v, v']_R) - r_\triangleleft (r_\triangleleft (h, v), v') - l_\triangleright (v, r_\triangleleft (h, v')) = 0.
\end{align*}
This shows that $(\mathfrak{h}, l_\triangleright, r_\triangleleft)$ is a representation of the Leibniz algebra $V_R$.
\end{proof}

We will use this representation of the Leibniz algebra $V_R$ in the description of the cohomology of the rRB Leibniz algebra $V \xrightarrow{R} \mathfrak{g}$ (see Section \ref{sec-5}).

\begin{remark} \cite[Theorem 2.7]{tang-leibniz} Let $V \xrightarrow{R} \mathfrak{g}$ be a rRB Leibniz algebra. Then the vector space $\mathfrak{g}$ is a representation of the induced Leibniz algebra $V_R$ with left and right $V_R$-actions given by
\begin{align*}
l_\triangleright (v, x) = [R(v), x]_\mathfrak{g} - R (r_V (v, x)) ~~~ \text{ and } ~~~ r_\triangleleft (x, v) = [x, R(v) ]_\mathfrak{g} - R (l_V (x, v)), ~ \text{ for } v \in V_R, x \in \mathfrak{g}.
\end{align*}
\end{remark}

\medskip

Given a Leibniz algebra and a representation of it, one can construct their semidirect product \cite{loday-pira}. In the following result, we generalize this result to rRB Leibniz algebras.

\begin{thm}\label{thm-semi}
Let $V \xrightarrow{R} \mathfrak{g}$ be a rRB Leibniz algebra and $(W \xrightarrow{S} \mathfrak{h}, l, r)$ be a representation of it. Then
\begin{itemize}
\item[(i)] the direct sum $\mathfrak{g} \oplus \mathfrak{h}$ carries a Leibniz algebra structure (denoted by $\mathfrak{g} \ltimes \mathfrak{h}$) with the bracket
\begin{align*}
[(x, h), (y, k)]_{\mathfrak{g} \oplus \mathfrak{h}} := \big( [x, y]_\mathfrak{g}, ~ l_\mathfrak{h} (x, k) + r_\mathfrak{h} (h, y) \big), ~ \text{ for } (x, h), (y, k) \in \mathfrak{g} \oplus \mathfrak{h}.
\end{align*} 

\item[(ii)] The vector space $V \oplus W$ can be equipped with a representation of the Leibniz algebra $\mathfrak{g} \ltimes \mathfrak{h}$ with left and right action maps
\begin{align*}
l_\ltimes  ( (x,h), (v, w)) :=~& (l_V (x, v), ~ l_W (x, w) + r (h, v)),\\
r_\ltimes ((v,w), (x,h)) :=~& (r_V (v,x),~ l(v, h) + r_W (w,x)).
\end{align*}

\item[(iii)] With the above structures, the map $R \oplus S  : V \oplus W \rightarrow \mathfrak{g} \oplus \mathfrak{h}$ is a relative Rota-Baxter operator. In other words, $V \oplus W \xrightarrow{R \oplus S} \mathfrak{g} \oplus \mathfrak{h}$ is a rRB Leibniz algebra. This is called the semidirect product.

\item[(iv)] Consider the inclusion maps $i_\mathfrak{g} : \mathfrak{g} \hookrightarrow \mathfrak{g} \oplus \mathfrak{h}$ and $i_V : V \hookrightarrow V \oplus W$. Then the pair $(i_\mathfrak{g}, i_V)$ is a morphism of rRB Leibniz algebras from $V \xrightarrow{R} \mathfrak{g}$ to the semidirect product $V \oplus W \xrightarrow{R \oplus S} \mathfrak{g} \oplus \mathfrak{h}$. 
\end{itemize}
\end{thm}

\begin{proof}
The part (i) is a standard result \cite{loday-pira}. Next, we observe that
\begin{align*}
&l_\ltimes \big( (x,h), l_\ltimes ((y,k), (v,w))   \big) \\
&= l_\ltimes \big(    (x,h), (l_V (y,v),~ l_W (y,w) + r (k, v)) \big) \\
&= \big( l_V (x, l_V (y,v)), ~ l_W (x, l_W (y, w)) + l_W (x, r(k, v)) + r (h, l_V (y, v))    \big) \\
&= \big(  l_V ([x,y]_\mathfrak{g}, v) ,~ l_W ([x,y]_\mathfrak{g}, w) + r (l_\mathfrak{h} (x,k), v) + r (r_\mathfrak{h} (h,y), v)  \big) \\
& ~~~~ + \big(  l_V (y, l_V (x,v)),~ l_W (y, l_W (x,w)) + l_W (y, r (h, v)) + r (k, l_V (x,v))   \big) \\
&= l_\ltimes \big(  ([x,y]_\mathfrak{g}, ~l_\mathfrak{h} (x,k) + r_\mathfrak{h} (h, y)), (v, w)  \big) + l_\ltimes \big( (y,k), (l_V (x,v),~ l_W (x, w) + r (h, v))   \big) \\
&= l_\ltimes \big(  [(x,h), (y, k)]_{\mathfrak{g} \oplus \mathfrak{h}}, (v,w) \big) + l_\ltimes \big( (y,k), l_\ltimes ((x,h), (v,w))  \big).
\end{align*}
Similarly, one can verify that
\begin{align*}
l_\ltimes \big( (x,h), r_\ltimes ((v,w), (y, k)) \big) =~& r_\ltimes \big( l_\ltimes ((x,h), (v,w)), (y, k)  \big) + r_\ltimes \big( (v,w), [(x,h), (y, k)]_{\mathfrak{g} \oplus \mathfrak{h}}  \big),\\
r_\ltimes \big( (v,w), [(x,h), (y, k)]_{\mathfrak{g} \oplus \mathfrak{h}}   \big) =~& r_\ltimes \big(  r_\ltimes ((v,w), (x,h)), (y,k)  \big) + l_\ltimes \big( (x,h), r_\ltimes ((v,w), (y, k))   \big).
\end{align*}
This shows that $V \oplus W$ is a representation of the Leibniz algebra $\mathfrak{g} \ltimes \mathfrak{h}$. This completes the proof of part (ii). Moreover, we see that
\begin{align*}
&[ (R \oplus S) (v,w), (R \oplus S) (v', w') ]_{\mathfrak{g} \oplus \mathfrak{h}} \\
&= [ (R(v), S(w)), (R(v'), S(w')) ]_{\mathfrak{g} \oplus \mathfrak{h}} \\
&= \big( [ R(v), R(v')]_\mathfrak{g} , ~ l_\mathfrak{h} (R(v), S(w')) + r_\mathfrak{h} (S(w), R(v'))   \big) \\
&= \big(  R (l_V (R(v), v') + r_V (v, R(v') ) ), ~ S (  l_W ( R(v), w') + l (v, S (w')) ) + S ( r ( S(w), v') + r_W (w, R(v')))  \big) \\
&= (R \oplus S) \big(  \big(  l_V (R(v), v'), ~ l_W (R(v), w') + r (S(w), v')  \big)   + \big(  r_V (v, R(v')), ~ l (v, S(w')) + r_W (w, R(v') ) \big)   \big) \\
&= (R \oplus S) \big(    l_\ltimes (  (R(v), S(w)) , (v', w')) + r_\ltimes ( (v,w), ( R(v'), S (w')) \big).
\end{align*}
which shows that $V \oplus W \xrightarrow{R \oplus S} \mathfrak{g} \oplus \mathfrak{h}$ is a rRB Leibniz algebra. Hence we have proved part (iii). Finally,  the part (iv) follows as $i_\mathfrak{g} : \mathfrak{g} \rightarrow \mathfrak{g} \oplus \mathfrak{h}$ is a morphism of Leibniz algebras, the map $i_V$ satisfies $i_V (l_V (x, v)) = l_\ltimes (i_\mathfrak{g} (x), i_V (v))$ and $i_V (r_V (v,x)) = r_\ltimes (i_V (v), i_\mathfrak{g} (x))$, and $i_\mathfrak{g} \circ R = (R \oplus S) \circ i_V$.
\end{proof}

\section{Cohomology of rRB Leibniz algebras}\label{sec-5}
In this section, we introduce the cohomology of an rRB Leibniz algebra with coefficients in a representation. Some applications of this cohomology are given in the next section.

\medskip

\noindent {\bf Cohomology with coefficients in the adjoint representation.}  Let $V \xrightarrow{R} \mathfrak{g}$ be a rRB Leibniz algebra, where $\mathfrak{g} = (\mathfrak{g}, \mu_\mathfrak{g})$ is a Leibniz algebra, $V = (V, l_V, r_V)$ is a representation and $R$ is a relative Rota-Baxter operator. Take the element $\pi = \mu_\mathfrak{g} + l_V + r_V \in (L')_1 = C^{1|0} (\mathfrak{g} \oplus V, \mathfrak{g} \oplus V)$ which we consider as an element $\pi [1] \in (L'[1])_0$. We have seen in Theorem \ref{l-inf-mc} that $\theta = (\pi [1], R) \in (L' [1] \oplus \mathfrak{a})_0$ is a Maurer-Cartan element in the $L_\infty$-algebra $(L'[1] \oplus \mathfrak{a}, \{ l_k \}_{ k \geq 1})$. Therefore, we can consider the $L_\infty$-algebra $(L'[1] \oplus \mathfrak{a}, \{ l_k^{(\pi[1], R)} \}_{ k \geq 1})$ twisted by $\theta = (\pi[1], R)$.

We will now define the cohomology of the rRB Leibniz algebra $V \xrightarrow{R} \mathfrak{g}$ (with coefficients in the adjoint representation). For each $n \geq 0$, we define an abelian group $C^n_\mathrm{rRB}(V \xrightarrow{R} \mathfrak{g})$ by
\begin{align*}
C^n_\mathrm{rRB} (V \xrightarrow{R} \mathfrak{g}) = \begin{cases}   0 & \text{ if } n =0,\\
\mathrm{Hom} (\mathfrak{g}, \mathfrak{g}) \oplus \mathrm{Hom} (V, V) & \text{ if } n =1,\\
\mathrm{Hom}(\mathfrak{g}^{\otimes n}, \mathfrak{g}) \oplus \mathrm{Hom}(\mathfrak{g}^{n-1, 1}, V) \oplus \mathrm{Hom}(V^{\otimes n-1}, \mathfrak{g}) & \text{ if } n \geq 2. \end{cases}
\end{align*}
Observe that an element $(\kappa, \eta) \in C^1_\mathrm{rRB} (V \xrightarrow{R} \mathfrak{g})$ gives rise to an element $( (\kappa + \eta)[1], 0) \in (L'[1] \oplus \mathfrak{a})_{-1}$. Moreover, the space $C^n_\mathrm{rRB}(V \xrightarrow{R} \mathfrak{g})$, for $n \geq 2$, is isomorphic to $(L'[1] \oplus \mathfrak{a})_{n-2}$ by
\begin{align*}
\mathrm{Hom}(\mathfrak{g}^{\otimes n}, \mathfrak{g}) \oplus \mathrm{Hom}(\mathfrak{g}^{n-1, 1}, V) \oplus \mathrm{Hom}(V^{\otimes n-1}, \mathfrak{g}) \ni   (\alpha, \beta, \gamma ) \leftrightsquigarrow ((\alpha + \beta)[1], \gamma) \in (L'[1] \oplus \mathfrak{a})_{n-2}.
\end{align*}

We define a map $\delta_\mathrm{rRB} : C^n_\mathrm{rRB}(V \xrightarrow{R} \mathfrak{g}) \rightarrow C^{n+1}_\mathrm{rRB}(V \xrightarrow{R} \mathfrak{g})$ by
\begin{align*}
\delta_\mathrm{rRB} ((\alpha, \beta, \gamma )) = (-1)^{n-2}~ l_1^{(\pi[1], R)} ((\alpha + \beta)[1], \gamma),~ \text{ for } (\alpha, \beta, \gamma) \in C^{n \geq 1}_\mathrm{rRB}(V \xrightarrow{R} \mathfrak{g}).
\end{align*}
%The explicit description of the map $\delta_\mathrm{rRB}$ will be given later when we will consider the cohomology with coefficients in a representation. 
Up to some scalar coefficient, the map $\delta_\mathrm{rRB}$ is the first map $l_1^{(\pi[1], R)}$ of the $L_\infty$-algebra $(L'[1] \oplus \mathfrak{a}, \{ l_k^{(\pi[1], R)} \}_{ k \geq 1})$ twisted by $ (\pi[1], R)$. This implies that $(\delta_\mathrm{rRB})^2 = 0$. In other words, $\{ C^\bullet_\mathrm{rRB} (V \xrightarrow{R} \mathfrak{g}), \delta_\mathrm{rRB} \}$ is a cochain complex.
%\begin{prop}
%We have $(\delta_\mathrm{rRB})^2 = 0$.
%\end{prop}
%\begin{proof}
%Since $(L'[1] \oplus \mathfrak{a}, \{ l_k^{(\pi [1], R)} \}_{k \geq 1})$ is an $L_\infty$-algebra, it follows that $l_1^{(\pi [1], R)} \circ l_1^{(\pi [1], R)} = 0$. Hence for $(\alpha, \beta, \gamma) \in C^n_\mathrm{rRB} (V \xrightarrow{R} \mathfrak{g})$ with $n \geq 1$,
%\begin{align*}
%(\delta_\mathrm{rRB})^2 ((\alpha, \beta, \gamma)) = (-1)^{n-2} (-1)^{n-1} ~l_1^{(\pi [1], R)} \circ l_1^{(\pi [1], R)} ((\alpha + \beta)[1], \gamma) = 0.
%\end{align*}
%Hence the proof.
%\end{proof}
Let $Z^n_\mathrm{rRB} (V \xrightarrow{R} \mathfrak{g})$ be the space of $n$-cocycles and $B^n_\mathrm{rRB} (V \xrightarrow{R} \mathfrak{g})$ be the space of $n$-coboundaries. The corresponding cohomology groups are called the cohomology of the rRB Leibniz algebra $V \xrightarrow{R} \mathfrak{g}$ with coefficients in itself. They are denoted by $H^\bullet_\mathrm{rRB} (V \xrightarrow{R} \mathfrak{g})$.
\medskip

\noindent {\bf Cohomology with coefficients in a representation.} Let $V \xrightarrow{R} \mathfrak{g}$ be a rRB Leibniz algebra and $(W \xrightarrow{S} \mathfrak{h}, l, r)$ be a representation of it. Consider the semidirect product rRB Leibniz algebra $V \oplus W \xrightarrow{R \oplus S} \mathfrak{g} \oplus \mathfrak{h}$ given in Theorem \ref{thm-semi}. For each $n \geq 0$, we define an abelian group $C^n_\mathrm{rRB} (V \xrightarrow{R} \mathfrak{g} ; W \xrightarrow{S} \mathfrak{h})$ by
\begin{align*}
C^n_\mathrm{rRB} (V \xrightarrow{R} \mathfrak{g} ; W \xrightarrow{S} \mathfrak{h}) = \begin{cases}  0 & \text{ if } n =0,\\
\mathrm{Hom} (\mathfrak{g}, \mathfrak{h}) \oplus \mathrm{Hom} (V,W) & \text{ if } n =1,\\
\mathrm{Hom}(\mathfrak{g}^{\otimes n}, \mathfrak{h}) \oplus \mathrm{Hom}(\mathfrak{g}^{n-1, 1}, W) \oplus \mathrm{Hom}(V^{\otimes n-1}, \mathfrak{h}) & \text{ if } n \geq 1. \end{cases}
\end{align*}
Note that an element $(\alpha, \beta, \gamma) \in C^n_\mathrm{rRB} (V \xrightarrow{R} \mathfrak{g} ; W \xrightarrow{S} \mathfrak{h})$ can be lifted to an element $\widetilde{(\alpha, \beta, \gamma)} \in C^n_\mathrm{rRB} (V \oplus W \xrightarrow{R \oplus S } \mathfrak{g} \oplus \mathfrak{h})$ by
\begin{align*}
\widetilde{(\alpha, \beta, \gamma)} = (\widehat{\alpha}, \widehat{\beta}, \widehat{\gamma}),
\end{align*}
where $~\widehat{}~$ stands the horizontal lift defined in (\ref{h-lift}). Observe that the element $(\alpha, \beta, \gamma)$ can be obtained from $\widetilde{(\alpha, \beta, \gamma)}$ just by restricting it to $C^n_\mathrm{rRB} (V \xrightarrow{R} \mathfrak{g} ; W \xrightarrow{S} \mathfrak{h})$. Moreover, $\widetilde{(\alpha, \beta, \gamma)} = 0$ implies that $(\alpha, \beta, \gamma) = 0$. Finally, the differential $\delta_\mathrm{rRB} (\widetilde{(\alpha, \beta, \gamma)})$ restricts to an element in ${C^{n+1}_\mathrm{rRB} (V \xrightarrow{R} \mathfrak{g}; W \xrightarrow{S} \mathfrak{h})}$. We define a map $\delta_\mathrm{rRB}' : C^n_\mathrm{rRB} (V \xrightarrow{R} \mathfrak{g} ; W \xrightarrow{S} \mathfrak{h}) \rightarrow C^{n+1}_\mathrm{rRB} (V \xrightarrow{R} \mathfrak{g} ; W \xrightarrow{S} \mathfrak{h})$ by
\begin{align*}
\delta'_\mathrm{rRB} ((\alpha, \beta, \gamma)) = \delta_\mathrm{rRB} ( \widetilde{(\alpha, \beta, \gamma)})|_{C^{n+1}_\mathrm{rRB} (V \xrightarrow{R} \mathfrak{g}; W \xrightarrow{S} \mathfrak{h})}.
\end{align*}
Here $\delta_\mathrm{rRB}$ is the coboundary operator of the rRB Leibniz algebra $V \oplus W \xrightarrow{R \oplus S} \mathfrak{g} \oplus \mathfrak{h}$ with coefficients in the adjoint representation. We observe that $\widetilde{   \delta'_\mathrm{rRB} ((\alpha, \beta, \gamma))} = \delta (\widetilde{(\alpha, \beta, \gamma)})$. Hence
\begin{align*}
\widetilde{(\delta'_\mathrm{rRB})^2 ((\alpha, \beta, \gamma))} = \delta_\mathrm{rRB} \big( \widetilde{\delta'_\mathrm{rRB} ((\alpha, \beta, \gamma))} \big) = \delta_\mathrm{rRB} (   \widetilde{(\alpha, \beta, \gamma)}) = 0
\end{align*}
which implies that $(\delta_\mathrm{rRB}')^2 = 0$. In the following, we write the explicit description of the coboundary operator $\delta'_\mathrm{rRB}$. First, for any $\alpha \in \mathrm{Hom}(\mathfrak{g}^{\otimes n}, \mathfrak{h})$, we define a map
$\delta^\alpha_{\mathfrak{g}, W} : \mathrm{Hom} (\mathfrak{g}^{n-1,1}, W) \rightarrow \mathrm{Hom}(\mathfrak{g}^{n,1}, W)$ by
\begin{align*}
(\delta^\alpha_{\mathfrak{g}, W} (\beta)) (x_1, \ldots, x_{n+1}) =& \sum_{i=1}^n (-1)^{i+1} (l+l_W) \big( x_i, (\alpha + \beta) (x_1, \ldots, \widehat{x_i}, \ldots, x_{n+1}) \big) \\
+& (-1)^{n+1} (r + r_W) \big( (\alpha + \beta) (x_1, \ldots, x_n), x_{n+1}  \big) \\
+& \sum_{1 \leq i < j \leq n+1} (-1)^{i} ~ \beta \big(   x_1, \ldots, \widehat{x_i}, \ldots, x_{j-1}, (\mu_\mathfrak{g} + l_V + r_V)(x_i, x_j), x_{j+1}, \ldots, x_{n+1} \big),
\end{align*} 
for $x_1, \ldots, \widehat{x_p}, \ldots, x_{n+1} \in \mathfrak{g}$ and $x_p \in V$ with $1 \leq p \leq n+1$. Moreover, let $\delta_{V, \mathfrak{h}} : \mathrm{Hom} (V^{\otimes n-1}, \mathfrak{h}) \rightarrow \mathrm{Hom} (V^{\otimes n}, \mathfrak{h})$ be the coboundary operator of the Leibniz algebra $V_R$ with coefficients in the representation $\mathfrak{h}$, given in Proposition \ref{vr-h}. Explicitly,
\begin{align*}
(\delta_{V, \mathfrak{h}} (\gamma)) (v_1, \ldots, v_n ) =& \sum_{i=1}^{n-1} (-1)^{i+1} \big(   l_\mathfrak{h} (R(v_i),\gamma (v_1, \ldots, \widehat{v_i}, \ldots, v_n)) - S \circ l (v_i,\gamma (v_1, \ldots, \widehat{v_i}, \ldots, v_n) ) \big) \\
+& (-1)^{n}~ r_\mathfrak{h} (\gamma (v_1, \ldots, v_{n-1}), R(v_n)) - (-1)^{n}~ S \circ r ( \gamma (v_1, \ldots, v_{n-1}), v_n) \\
+& \sum_{1 \leq i < j \leq n} (-1)^i~ \gamma \big( v_1, \ldots, \widehat{v_i}, \ldots, v_{j-1},  l_V (R(v_i), v_j) + r_V (v_i, R(v_j)), v_{j+1}, \ldots, v_n   \big),
\end{align*}
for $v_1, \ldots, v_n \in V$. Finally, we consider another map $h_R : \mathrm{Hom}(\mathfrak{g}^{\otimes n} , \mathfrak{h}) \oplus \mathrm{Hom} (\mathfrak{g}^{n-1, 1}, W) \rightarrow \mathrm{Hom}(V^{\otimes n}, \mathfrak{h})$ by
\begin{align*}
(h_R (\alpha, \beta)) (v_1, \ldots, v_n ) = (-1)^n \big\{ \alpha (R(v_1), \ldots, R(v_n)) - \sum_{i=1}^n S (\beta ( R(v_1), \ldots, v_i, \ldots, R(v_n))   )   \big\}.
\end{align*}
With all these maps, the map $\delta'_\mathrm{rRB} : C^n_\mathrm{rRB} (V \xrightarrow{R} \mathfrak{g}; W \xrightarrow{S} \mathfrak{h}) \rightarrow C^{n+1}_\mathrm{rRB} (V \xrightarrow{R} \mathfrak{g}; W \xrightarrow{S} \mathfrak{h})$ is given by
\begin{align*}
\delta'_\mathrm{rRB}  ((\alpha, \beta, \gamma)) = \big( \delta_{\mathfrak{g}, \mathfrak{h}} (\alpha), \delta^\alpha_{\mathfrak{g}, W} (\beta), \delta_{V, \mathfrak{h}} (\gamma) + h_R (\alpha, \beta)   \big).
\end{align*}
It follows from the above discussions that $\{ C^\bullet_\mathrm{rRB} (V \xrightarrow{R} \mathfrak{g}; W \xrightarrow{S} \mathfrak{h}), \delta'_\mathrm{rRB} \}$ is a cochain complex. The corresponding cohomology groups are called the cohomology of the rRB Leibniz algebra $V \xrightarrow{R} \mathfrak{g}$ with coefficients in the representation $(W \xrightarrow{S} \mathfrak{h}, l, r)$, and they are denoted by $H^\bullet_\mathrm{rRB}   ( V \xrightarrow{R} \mathfrak{g}; W \xrightarrow{S} \mathfrak{h}).$

\begin{remark}
Let $(\mathfrak{g},R)$ be a Rota-Baxter Leibniz algebra and $(V, R_V)$ be a representation. Then we have seen in Remark \ref{rmk-rb} and Proposition \ref{rpr-rb} that $\mathfrak{g} \xrightarrow{R} \mathfrak{g}$ is a rRB Leibniz algebra and $(V \xrightarrow{R_V} V, l_V, r_V)$ is a representation of it. One may define the cohomology of the Rota-Baxter Leibniz algebra $(\mathfrak{g}, R)$ with coefficients in $(V, R_V)$ as the cohomology of the rRB Leibniz algebra $\mathfrak{g} \xrightarrow{R} \mathfrak{g}$ with coefficients in $(V \xrightarrow{R_V} V, l_V, r_V)$.
\end{remark}

%\textcolor{red}{Lower degree interpretation}

\section{Applications of cohomology}\label{sec-6}
In this section, we study deformations and abelian extensions of rRB Leibniz algebras in terms of cohomology. In particular, we show that (i) the equivalence classes of infinitesimal deformations of an rRB Leibniz algebra $V \xrightarrow{R} \mathfrak{g}$ are in one-to-one correspondence with the second cohomology group $H^2_\mathsf{rRB} (V \xrightarrow{R} \mathfrak{g})$ of the rRB Leibniz algebra $V \xrightarrow{R} \mathfrak{g}$ with coefficients in the adjoint representation, and (ii) the isomorphism classes of abelian extensions of $V \xrightarrow{R} \mathfrak{g}$ by a given representation are in one-to-one correspondence with the second cohomology group with coefficients in the representation.

\medskip

\noindent {\bf Deformations of rRB Leibniz algebras.} Let $\mathsf{R}$ be an augmented unital ring with an augmentation $\epsilon : \mathsf{R} \rightarrow {\bf k}$. Note that, replacing vector spaces by modules over $\mathsf{R}$ and linear maps over {\bf k} by $\mathsf{R}$-linear maps in Definitions  \ref{rrbl}  and \ref{rrbl-mor}, one can easily define rRB Leibniz algebras over $\mathsf{R}$ and (iso)morphisms between them. Note that any rRB Leibniz algebra $V \xrightarrow{R} \mathfrak{g}$ can be regarded as a rRB Leibniz algebra over $\mathsf{R}$, where the $\mathsf{R}$-module structures on $\mathfrak{g}$ and $V$ are respectively given by $r \cdot x = \epsilon (r) x$ and $r \cdot v = \epsilon (r) v$, for $r \in \mathsf{R}$, $x \in \mathfrak{g}$ and $v \in V$.

\begin{defn}\label{r-def}
An $\mathsf{R}$-deformation of a rRB Leibniz algebra $V \xrightarrow{R} \mathfrak{g}$ consists of a quadruple $(\mu_\mathsf{R}, l_\mathsf{R}, r_\mathsf{R}, R_\mathsf{R})$ of $\mathsf{R}$-linear maps
\begin{align*}
\mu_\mathsf{R} : (\mathsf{R} \otimes_{\bf k} \mathfrak{g}) \otimes (\mathsf{R} \otimes_{\bf k} \mathfrak{g}) &\rightarrow \mathsf{R} \otimes_{\bf k} \mathfrak{g}, \quad l_\mathsf{R} : (\mathsf{R} \otimes_{\bf k} \mathfrak{g}) \otimes (\mathsf{R} \otimes_{\bf k} V) \rightarrow \mathsf{R} \otimes_{\bf k} V \\
&r_\mathsf{R} : (\mathsf{R} \otimes_{\bf k} V) \otimes (\mathsf{R} \otimes_{\bf k} \mathfrak{g}) \rightarrow \mathsf{R} \otimes_{\bf k} V
\end{align*}
and an $\mathsf{R}$-linear map $R_\mathsf{R} : \mathsf{R} \otimes_{\bf k} V \rightarrow \mathsf{R} \otimes_{\bf k} \mathfrak{g}$ that makes $( \mathsf{R} \otimes_{\bf k} \mathfrak{g}, \mu_\mathsf{R}  )$ into a Leibniz algebra over $\mathsf{R}$, $(\mathsf{R} \otimes_{\bf k} V, l_\mathsf{R}, r_\mathsf{R})$ a representation and $R_\mathsf{R} : \mathsf{R} \otimes_{\bf k} V \rightarrow \mathsf{R} \otimes_{\bf k} \mathfrak{g}$ a relative Rota-Baxter operator. In other words, $\mathsf{R} \otimes_{\bf k} V \xrightarrow{R_\mathsf{R}} \mathsf{R} \otimes_{\bf k} \mathfrak{g}$ is a rRB Leibniz algebra over $\mathsf{R}$.
\end{defn}

\begin{defn}
Let $(\mu_\mathsf{R}, l_\mathsf{R}, r_\mathsf{R}, R_\mathsf{R})$ and $(\mu'_\mathsf{R}, l'_\mathsf{R}, r'_\mathsf{R}, R'_\mathsf{R})$ be two $\mathsf{R}$-deformations of a rRB Leibniz algebra $V \xrightarrow{R} \mathfrak{g}$. They are said to be equivalent if there exists an isomorphism
\begin{align*}
(\Phi, \Psi) : (\mathsf{R} \otimes_{\bf k} V \xrightarrow{R_\mathsf{R}} \mathsf{R} \otimes_{\bf k} \mathfrak{g}) \rightsquigarrow (\mathsf{R} \otimes_{\bf k} V \xrightarrow{R'_\mathsf{R}} \mathsf{R} \otimes_{\bf k} \mathfrak{g})
\end{align*}
of rRB Leibniz algebras over $\mathsf{R}$ satisfying $(\epsilon \otimes_\mathbf{k} \mathrm{id}_\mathfrak{g}) \circ \Phi = (\epsilon \otimes_\mathbf{k} \mathrm{id}_\mathfrak{g}) $ and $(\epsilon \otimes_\mathbf{k} \mathrm{id}_V) \circ \Psi = (\epsilon \otimes_\mathbf{k} \mathrm{id}_V)$.
\end{defn}

In the following, we will be interested in $\mathsf{R}$-deformations of rRB Leibniz algebras, when $\mathsf{R} = \mathbf{k}[[t]]$ (the ring of formal power series) or $\mathsf{R} = \mathbf{k}[[t]]/(t^2)$ (the local Artinian ring of dual numbers).

\begin{defn}
A formal deformation of a rRB Leibniz algebra $V \xrightarrow{R} \mathfrak{g}$ is a $\mathsf{R}$-deformation in the sense of Definition \ref{r-def}, where $\mathsf{R} = \mathbf{k}[[t]]$.
\end{defn}

Thus, a formal deformation of a rRB Leibniz algebra $V \xrightarrow{R} \mathfrak{g}$ consists of a quadruple $(\mu_t, l_t, r_t, R_t)$ of formal sums
\begin{align*}
\mu_t = \sum_{i=0}^\infty \mu_i t^i, \quad l_t = \sum_{i=0}^\infty l_it^i, \quad r_t = \sum_{i=0}^\infty r_it^i ~~~ \text{ and } ~~~ R_t = \sum_{i=0}^\infty R_i t^i,
\end{align*}
(where $\mu_i \in \mathrm{Hom}(\mathfrak{g}^{\otimes 2}, \mathfrak{g})$, $l_i \in \mathrm{Hom}(\mathfrak{g} \otimes V, V)$, $r_i \in \mathrm{Hom}(V \otimes \mathfrak{g}, V)$, $R_i \in \mathrm{Hom}(V, \mathfrak{g})$, for $ i \geq 0$, with $\mu_0 = \mu_\mathfrak{g}$, $l_0 = l_V$, $r_0 = r_V$ and $R_0 = R$) such that $\mathfrak{g}[[t]] = (\mathfrak{g}[[t]], \mu_t)$ is a Leibniz algebra over $\mathbf{k}[[t]]$, $V[[t]] = (V[[t]], l_t, r_t)$ is a representation of the Leibniz algebra $\mathfrak{g}[[t]]$ and $R_t : V[[t]] \rightarrow \mathfrak{g}[[t]]$ is a relative Rota-Baxter operator. In other words, $V[[t]] \xrightarrow{R_t} \mathfrak{g}[[t]]$ is a rRB Leibniz algebra over $\mathbf{k}[[t]]$.

Two formal deformations $(\mu_t, l_t, r_t, R_t)$ and $(\mu'_t, l'_t, r'_t, R'_t)$ are said to be equivalent if there are formal sums
\begin{align*}
\phi_t = \sum_{i=0}^\infty \phi_i t^i, ~~~~ 
\psi_i = \sum_{i=0}^\infty \psi_i t^i ~~~~ \quad (\mathrm{with }~ \phi_i \in \mathrm{Hom}(\mathfrak{g}, \mathfrak{g}), \psi_i \in \mathrm{Hom}(V, V) \text{ with } \phi_0 = \mathrm{id}_\mathfrak{g}, \psi_0 = \mathrm{id}_V)
\end{align*}
such that $(\phi_t, \psi_t) : (V[[t]] \xrightarrow{R_t} \mathfrak{g}[[t]]) \rightsquigarrow  (V[[t]] \xrightarrow{R'_t} \mathfrak{g}[[t]])$ is a morphism of rRB Leibniz algebras  over $\mathbf{k}[[t]]$.

\medskip

Let $(\mu_t, l_t, r_t, R_t)$ be a formal deformation of a rRB Leibniz algebra $V \xrightarrow{R} \mathfrak{g}$. Then for any $x, y, z \in \mathfrak{g}$, $v, v' \in V$ and $n \geq 0$, we have
\begin{align*}
\sum_{i+j = n} \mu_i (x, \mu_j (y, z) ) =~& \sum_{i+j = n} \big( \mu_i (\mu_j (x, y), z) + \mu_i (y, \mu_j (x,z))   \big),\\
\sum_{i+j = n}  l_i (x, l_j (y, v)) =~& \sum_{i+j = n}  \big(  l_i (\mu_j (x,y), v) + l_i (y, l_j (x, v)) \big),\\
\sum_{i+j = n}  l_i (x, r_j (v, y)) =~& \sum_{i+j = n}  \big(  r_i (l_j (x,v), y) + r_i (v, \mu_j (x, y))  \big),\\
\sum_{i+j = n}  r_i (v, \mu_j (x,y)) =~& \sum_{i+j = n}  \big( r_i (r_j (v,x), y) + l_i (x, r_j (v,y))  \big),\\
\sum_{i+j + k = n}  \mu_i (R_j (v), R_k (v')) =~& \sum_{i+j+k = n}  R_i \big( l_j (R_k (v), v') + r_j (v, R_k (v'))   \big).
\end{align*}
Note that all these identities are hold for $n =0$. To summarize all these identities for $n=1$, we first define an element $\beta_1 \in \mathrm{Hom}(\mathfrak{g}^{1,1}, V)$ by
\begin{align}\label{beta1}
\beta_1 (x,v) = l_1 (x,v) ~~~~ \text{ and } ~~~~ \beta_1 (v,x) = r_1 (v,x), ~ \text{ for } x \in \mathfrak{g}, v \in V.
\end{align}
Then for $n=1$, we get
\begin{align*}
\delta_\mathsf{rRB} (\mu_1, \beta_1, R_1) = 0.
\end{align*}
In other words, $(\mu_1, \beta_1, R_1) \in Z^2_\mathsf{rRB} (V \xrightarrow{R} \mathfrak{g})$ is a $2$-cocycle in the cohomology complex of the rRB Leibniz algebra $V \xrightarrow{R} \mathfrak{g}$ with coefficients in the adjoint representation.

Further, if $(\mu_t, l_t, r_t, R_t)$ and $(\mu'_t, l'_t, r'_t, R'_t)$ are two equivalent formal deformations via $(\phi_t, \psi_t)$, then for any $x, y \in \mathfrak{g}$, $v \in V$ and $n \geq 0$, we have
\begin{align*}
\sum_{i+j = n}  \phi_i (\mu_j (x, y)) =~& \sum_{i+j +k = n} \mu'_i (\phi_j (x), \phi_k (y)), \\
\sum_{i+j = n}  \psi_i (l_j (x, v) ) =~& \sum_{i+j+k = n} l'_i (\phi_j (x), \psi_k (v)), \\
\sum_{i+j = n}  \psi_i (r_j (v, x)) =~& \sum_{i+j + k= n} r'_i (\psi_j (v), \phi_k (x)), \\
\sum_{i+j = n} \phi_i \circ R_j =~& \sum_{i+j = n} R'_i \circ \psi_j.
\end{align*}
Combining all these identities for $n =1$, we simply get $(\mu_1, \beta_1, R_1) - (\mu'_1, \beta'_1, R'_1) = \delta_\mathsf{rRB} ((\phi_1 , \psi_1)).$ Therefore, we obtain a map 
\begin{align}\label{formal-def-2co}
(\text{formal deformations of }~ V \xrightarrow{R} \mathfrak{g}) / \sim \quad \rightarrow H^2_\mathrm{rRB} (V \xrightarrow{R} \mathfrak{g}).
\end{align}

\medskip

In the following, we consider a truncated version of formal deformations (called infinitesimal deformations) and generalize the map (\ref{formal-def-2co}) into an isomorphism.

\begin{defn}
An infinitesimal deformation of a rRB Leibniz algebra $V \xrightarrow{R} \mathfrak{g}$ is a $\mathsf{R}$-deformation in the sense of Definition \ref{r-def}, where $\mathsf{R} = \mathbf{k}[[t]]/(t^2)$.
\end{defn}

\begin{thm}\label{inf-def-2}
Let $V \xrightarrow{R} \mathfrak{g}$ be a rRB Leibniz algebra. Then there is a one-to-one correspondence between equivalence classes of infinitesimal deformations of $V \xrightarrow{R} \mathfrak{g}$, and the second cohomology group $H^2_\mathrm{rRB} (V \xrightarrow{R} \mathfrak{g})$.
\end{thm}

\begin{proof}
Let $(\mu_t = \mu_\mathfrak{g} + t \mu_1, l_t = l_V + t l_1, r_t = r_V + t r_1, R_t = R + t R_1)$ be an infinitesimal deformation of the rRB Leibniz algebra $V \xrightarrow{R} \mathfrak{g}$. Then similar to formal deformations, one can show that $(\mu_1, \beta_1, R_1) \in Z^2_\mathrm{rRB} (V \xrightarrow{R} \mathfrak{g})$ is a $2$-cocycle, where $\beta_1$ is given by (\ref{beta1}). Similarly, if two infinitesimal deformations are equivalent, then the corresponding $2$-cocycles are differ by a coboundary. Therefore, there is a well-defined map
\begin{align*}
\Theta_1 : (\text{infinitesimal deformations of }~ V \xrightarrow{R} \mathfrak{g}) / \sim \quad \rightarrow H^2_\mathrm{rRB} (V \xrightarrow{R} \mathfrak{g}).
\end{align*}

Conversely, let $(\mu_1, \beta_1, R_1) \in Z^2_\mathrm{rRB}(V \xrightarrow{R} \mathfrak{g})$ be a $2$-cocycle. Then it is easy to verify that the quadruple $(\mu_t = \mu_\mathfrak{g} + t \mu_1, l_t = l_V + t l_1, r_t = r_V + t r_1, R_t = R + t R_1)$ is an infinitesimal deformation of the rRB Leibniz algebra $V \xrightarrow{R} \mathfrak{g}$, where $l_1, r_1$ are defined from $\beta_1$ simply by (\ref{beta1}). Let $(\mu'_1, \beta'_1, R'_1) \in Z^2_\mathrm{rRB}(V \xrightarrow{R} \mathfrak{g})$ be another $2$-cocycle cohomologous to $(\mu_1, \beta_1, R_1)$, say $(\mu_1, \beta_1, R_1) - (\mu'_1, \beta'_1, R'_1) = \delta_\mathrm{rRB} ((\phi_1 , \psi_1))$. Then the corresponding infinitesimal deformations $(\mu_t, l_t, r_t, R_t)$ and $(\mu'_t, l'_t, r'_t, R'_t)$ are equivalent via $(\phi_t = \mathrm{id}_\mathfrak{g} + t \phi_1, \psi_t = \mathrm{id}_V + t \psi_1)$. Hence there is a well-defined map  $\Theta_2 : H^2_\mathrm{rRB} (V \xrightarrow{R} \mathfrak{g}) \rightarrow (\text{infinitesimal deformations of }~ V \xrightarrow{R} \mathfrak{g}) / \sim$. Finally, the maps $\Theta_1$ and $\Theta_2$ are inverses to each other.
\end{proof}

\medskip

\medskip

\noindent {\bf Abelian extensions of rRB Leibniz algebras.} Let $V \xrightarrow{R} \mathfrak{g}$ be a rRB Leibniz algebra and $W \xrightarrow{S} \mathfrak{h}$ be a $2$-term chain complex (not necessarily a representation). Note that $W \xrightarrow{S} \mathfrak{h}$ can be regarded as a rRB Leibniz algebra with the trivial Leibniz bracket on $\mathfrak{h}$ and the trivial $\mathfrak{h}$-representation on $W$.

\begin{defn}
An {\bf abelian extension} of $V \xrightarrow{R} \mathfrak{g}$ by the $2$-term chain complex $W \xrightarrow{S} \mathfrak{h}$ is a short exact sequence of rRB Leibniz algebras
\begin{align}\label{abelian-diag}
\xymatrix{
o \ar[r] & W \ar[r]^{\overline{i}} \ar[d]_S & \widehat{V} \ar[r]^{\overline{p}} \ar[d]^{\widehat{R}} & V \ar[r] \ar[d]^R & 0 \\
0 \ar[r] & \mathfrak{h} \ar[r]_i & \widehat{\mathfrak{g}}  \ar[r]_p & \mathfrak{g} \ar[r] & 0.
}
\end{align}
We denote an abelian extension as above simply by $\widehat{V} \xrightarrow{\widehat{R}} \widehat{\mathfrak{g}}$ when the structure maps in the diagram (\ref{abelian-diag}) are understood.
\end{defn}

A section of the abelian extension (\ref{abelian-diag}) is a pair $(s, \overline{s})$ of linear maps $s : \mathfrak{g} \rightarrow \widehat{\mathfrak{g}}$ and $\overline{s} : V \rightarrow \widehat{V}$ satisfying $p \circ s = \mathrm{id}_\mathfrak{g}$ and $\overline{p} \circ \overline{s} = \mathrm{id}_V$. A section of (\ref{abelian-diag}) always exists.

Let $(s, \overline{s})$ be a section of (\ref{abelian-diag}). We define linear maps $l_\mathfrak{h} : \mathfrak{g} \otimes \mathfrak{h} \rightarrow \mathfrak{h}$ and $r_\mathfrak{h} : \mathfrak{h} \otimes \mathfrak{g} \rightarrow \mathfrak{h}$ by
\begin{align*}
l_\mathfrak{h} (x,h) = [s(x), i(h)]_{\widehat{\mathfrak{g}}}   ~~~~ \text{ and } ~~~~ r_\mathfrak{h} (h, x) = [i(h), s(x) ]_{\widehat{\mathfrak{g}}}, ~\text{ for } x \in \mathfrak{g}, h \in \mathfrak{h}.
\end{align*}
It is easy to verify that the above two maps defines a representation of the Leibniz algebra $\mathfrak{g}$ on the vector space $\mathfrak{h}$. Similarly, the vector space $W$ can be equipped with a representation of the Leibniz algebra $\mathfrak{g}$ via $l_W : \mathfrak{g} \otimes W \rightarrow W$ and $r_W : W \otimes \mathfrak{g} \rightarrow W$ given by
\begin{align*}
l_W (x,w) = l_{\widehat{V}} (s(x), \overline{i}(w)) ~~~ \text{ and } ~~~ r_W (w, x) = r_{\widehat{V}} (\overline{i}(w), s(x)), ~ \text{ for } x \in \mathfrak{g}, w \in W.
\end{align*} 
Moreover, we define linear maps $l : V \otimes \mathfrak{h} \rightarrow W$ and $r : \mathfrak{h} \otimes V \rightarrow W$ by
\begin{align*}
l (v, h) = l_{\widehat{V}} (\overline{s}(v), i(h)) ~~~~ \text{ and } ~~~~ r (h, v) = r_{\widehat{ V}} (i(h), \overline{s}(v)), \text{ for } v \in V, h \in \mathfrak{h}.
\end{align*}
It is easy to verify that the maps $l, r$ satisfy the identities of (\ref{rep-1s}) and (\ref{rep-2s}). For any $v \in V$ and $w \in W$, we also have
\begin{align*}
l_\mathfrak{h} (R(v), S(w)) = [sR(v), iS (w)]_{\widehat{\mathfrak{g}}} =~& [\widehat{R}(\overline{s}(v)), \widehat{R}(\overline{i}(w)) ]_{\widehat{\mathfrak{g}}}\\
=~& \widehat{R} \big( l_{\widehat{V}} (   \widehat{R} (\overline{s} (v)), \overline{i}(w)) + r_{\widehat{V}} ( \overline{s}(v), \widehat{R} (\overline{i}(w)))   \big)\\
=~& \widehat{R} \big( l_{\widehat{V}} (   sR(v), \overline{i}(w)) + r_{\widehat{V}} ( \overline{s}(v), iS(w) )   \big) \\
=~& S \big( l_W (R(v), w) + l (v, S (w))   \big).
\end{align*}
Similarly, we can show that $r_\mathfrak{h} (S(w), R(v)) = S \big( r (S (w), v) + r_W (w, R(v) )  \big).$ Therefore, an abelian extension (\ref{abelian-diag}) induces a representation $(W \xrightarrow{S} \mathfrak{h}, l, r)$ of the rRB Leibniz algebra $V \xrightarrow{R} \mathfrak{g}$.

%\begin{remark}
Let $(s', \overline{s}')$ be any other section of the abelian extension (\ref{abelian-diag}). Then we have $s(x) - s'(x) \in \mathrm{ker} (p) = \mathrm{im}(i)$ and $\overline{s}(v) - \overline{s}'(v) \in \mathrm{ker} (\overline{p}) = \mathrm{im}(\overline{i}),$ for $x \in \mathfrak{g}$ and $v \in V$. Hence we have
\begin{align*}
& l_\mathfrak{h} (x, h) - l'_\mathfrak{h} (x, h) = [s(x) - s'(x), i(h)]_{\widehat{\mathfrak{g}}} = 0 ~~~ \text{ and } ~~~ r_\mathfrak{h} (h, x) - r'_\mathfrak{h} (h, x) = [i(h), s(x) - s'(x)]_{\widehat{\mathfrak{g}}} = 0,\\
 & l_W (x, w) - l'_W (x, w) = l_{\widehat{V}} (s(x) - s'(x) , \overline{i}(w)) = 0 ~~~ \text{ and } ~~~ r_W (w, x) - r'_W (w, x) = r_{\widehat{V}} (\overline{i} (w), s(x) - s'(x)) = 0,\\
 & l(v, h) - l' (v, h) = l_{\widehat{V}} (\overline{s}(v) - \overline{s}' (x) , i(h)) = 0 ~~~ \text{ and } ~~~ r (h, v) - r'(h, v) = r_{\widehat{V}} (i(h), \overline{s} (v) - \overline{s}'(v)) = 0.
\end{align*}
\big(Here $(l'_\mathrm{h}, r'_\mathfrak{h})$, $(l'_W, r'_W)$ and $(l', r')$ denote the structures induced by the section $(s', \overline{s}')$\big). This shows that the structure of the representation $(W \xrightarrow{S} \mathfrak{h}, l, r)$ is independent of the choice of a section of (\ref{abelian-diag}).
%\end{remark}

\begin{defn}\label{defn-iso-abel}
Let $V \xrightarrow{R} \mathfrak{g}$ be a rRB Leibniz algebra and $W \xrightarrow{S} \mathfrak{h}$ be a $2$-term chain complex. Two abelian extensions $\widehat{V} \xrightarrow{\widehat{R}} \widehat{\mathfrak{g}}$ and $\widehat{V}' \xrightarrow{\widehat{R}'} \widehat{\mathfrak{g}}'$ are said to be {\bf isomorphic} if there is an isomorphism $(\phi, \psi) : (\widehat{V} \xrightarrow{\widehat{R}} \widehat{\mathfrak{g}}) \rightsquigarrow (\widehat{V}' \xrightarrow{\widehat{R}'} \widehat{\mathfrak{g}}')$ of rRB Leibniz algebras making the following diagram commutative
\begin{align}
\xymatrixrowsep{0.36cm}
\xymatrixcolsep{0.36cm}
\xymatrix{
0 \ar[rr] &  & W \ar[rr] \ar[dd] \ar@{=}[rd] & & \widehat{V} \ar[rr] \ar[rd]^\psi \ar[dd] & & V \ar[dd] \ar[rr] \ar@{=}[rd] & & 0 \\
 & 0 \ar[rr] & & W \ar[rr] \ar[dd] & & \widehat{V}' \ar[rr] \ar[dd] & & V \ar[rr] \ar[dd] & & 0 \\
0 \ar[rr] &  & \mathfrak{h} \ar[rr] \ar@{=}[rd] & & \widehat{\mathfrak{g}} \ar[rr] \ar[rd]^\phi & & \mathfrak{g} \ar[rr] \ar@{=}[rd] & & 0 \\
 & 0 \ar[rr] & & \mathfrak{h} \ar[rr] & & \widehat{\mathfrak{g}}' \ar[rr] & & \mathfrak{g} \ar[rr] & & 0. \\
}
\end{align}
\end{defn}

Let $V \xrightarrow{R} \mathfrak{g}$ be a rRB Leibniz algebra and $(W \xrightarrow{S} \mathfrak{h}, l, r)$ be a representation of it (Definition \ref{repn-defn}). Let $\mathrm{Ext}(V \xrightarrow{R} \mathfrak{g}; W \xrightarrow{S} \mathfrak{h})$ denote the set of isomorphism classes of abelian extensions of $V \xrightarrow{R} \mathfrak{g}$ by the $2$-term chain complex $W \xrightarrow{S} \mathfrak{h}$ so that the induced representation coincides with the prescribed one. Then we have the following.

\begin{thm}\label{abelian-ext-2}
There is a one-to-one correspondence between $\mathrm{Ext}(V \xrightarrow{R} \mathfrak{g}; W \xrightarrow{S} \mathfrak{h})$ and the second cohomology group $H^2_\mathrm{rRB} (V \xrightarrow{R} \mathfrak{g}; W \xrightarrow{S} \mathfrak{h}).$
\end{thm}

\begin{proof}
Let (\ref{abelian-diag}) be an abelian extension of the rRB Leibniz algebra $V \xrightarrow{R} \mathfrak{g}$ by the $2$-term chain complex $W \xrightarrow{S} \mathfrak{h}$. For any section $(s, \overline{s})$, we define maps
\begin{align*}
&\alpha \in \mathrm{Hom} (\mathfrak{g}^{\otimes 2}, \mathfrak{h}), ~~~~ \alpha (x, y) = [s(x), s(y)]_{\widehat{\mathfrak{g}}} - s ([x, y]_\mathfrak{g}), \\
&\beta \in \mathrm{Hom} (\mathfrak{g}^{1,1}, W), ~~~~ \begin{cases}  \beta (x, v) = l_{\widehat{V}} (s(x), \overline{s}(v)) - \overline{s} (l_V (x, v)),\\
\beta (v,x) = r_{\widehat{V}} (\overline{s}(v), s(x)) - \overline{s} (r_V (v,x)),  
\end{cases}\\
&\gamma \in \mathrm{Hom}(V, \mathfrak{h}), ~~~~ \gamma (v) = \widehat{R} (\overline{s}(v)) - s (R(v)).
\end{align*}
Then similar to \cite{jiang-sheng}, one can show that the triple $(\alpha, \beta, \gamma)$ is a $2$-cocycle in $Z^2_{\mathrm{rRB}} (V \xrightarrow{R} \mathfrak{g}; W \xrightarrow{S} \mathfrak{h})$. Moreover, the corresponding cohomology class does not depend on the choice of the section.

Let $\widehat{V} \xrightarrow{\widehat{R}} \widehat{\mathfrak{g}}$ and  $\widehat{V}' \xrightarrow{\widehat{R}'} \widehat{\mathfrak{g}}'$ be two isomorphic abelian extensions as of Definition \ref{defn-iso-abel}. If $(s, \overline{s})$ is a section of the first abelian extension, then we have
$p' \circ (\phi \circ s) = p \circ s = \mathrm{id}_\mathfrak{g}$ and $\overline{p}' \circ (\psi \circ \overline{s}) = \overline{p} \circ \overline{s} = \mathrm{id}_V$. This shows that $(\phi \circ s, \psi \circ \overline{s})$ is a section of the second abelian extension. Let $(\alpha', \beta', \gamma')$ be the $2$-cocycle in $Z^2_{\mathrm{rRB}} (V \xrightarrow{R} \mathfrak{g}; W \xrightarrow{S} \mathfrak{h})$ corresponding to the second abelian extension $\widehat{V}' \xrightarrow{\widehat{R}'} \widehat{\mathfrak{g}}'$ and its section $(\phi \circ s, \psi \circ \overline{s})$. Then we have
\begin{align*}
\alpha' (x, y) =~& [ \phi \circ s (x), \phi \circ s (y)]_{\widehat{\mathfrak{g}}'} - (\phi \circ s) [x,y]_\mathfrak{g} \\
=~& \phi \big( [s(x), s(y)]_{\widehat{\mathfrak{g}}} - s[x,y]_\mathfrak{g}  \big) = \phi (\alpha (x,y)) = \alpha (x, y)  \quad (\because \phi|_\mathfrak{h} = \mathrm{id}_\mathfrak{h}).
\end{align*}
By similar observations, one can show that $ \beta' = \beta$ and $\gamma' = \gamma$. Hence $(\alpha' , \beta', \gamma') = (\alpha, \beta, \gamma)$. Therefore, there is a well-defined map 
\begin{align*}
\Theta_1 : \mathrm{Ext} (V \xrightarrow{R} \mathfrak{g}; W \xrightarrow{S} \mathfrak{h}) \rightarrow H^2_\mathrm{rRB} (V \xrightarrow{R} \mathfrak{g}; W \xrightarrow{S} \mathfrak{h}).
\end{align*}

\medskip

Conversely, let $(\alpha, \beta, \gamma) \in Z^2_{\mathrm{rRB}} (V \xrightarrow{R} \mathfrak{g}; W \xrightarrow{S} \mathfrak{h})$ be a $2$-cocycle. Take $\widehat{\mathfrak{g}} = \mathfrak{g} \oplus \mathfrak{h}$ and $\widehat{V} = V \oplus W$, and define some operations 
\begin{align*}
[(x, h), (y, k)]_{\widehat{\mathfrak{g}}} : = ([x,y]_\mathfrak{g}, ~ l_\mathfrak{h} (x, k) + r_\mathfrak{h} (h, y) + \alpha (x, y) ), \\
l_{\widehat{V}} ((x, h), (v, w)) := (l_V (x, v), ~ l_W (x, w) + r (h, v) + \beta (x, v) ),\\
r_{\widehat{V}} ((v, w), (x,h)) := (r_V (v, x),~ l (v, h) + r_W (w, x) + \beta (v, x) ),
\end{align*}
for $(x, h), (y, k) \in \widehat{\mathfrak{g}}$ and $(v, w) \in \widehat{V}$. It is easy to verify that $(\widehat{\mathfrak{g}}, [~,~]_{\widehat{\mathfrak{g}}})$ is a Leibniz algebra and $(\widehat{V}, l_{\widehat{V}}, r_{\widehat{V}})$ is a representation of it. We also define a map $\widehat{R}: \widehat{V} \rightarrow \widehat{\mathfrak{g}}$ by
\begin{align*}
\widehat{R}((v, w)) = (R(v), S (w) + \gamma (v)), ~ \text{ for } (v, w) \in \widehat{V}.
\end{align*}
Since $\gamma$ satisfies $\delta_{V, \mathfrak{h}} (\gamma) + h_R (\alpha, \beta) = 0$, we have that $\widehat{R}$ is a relative Rota-Baxter operator. In other words, $\widehat{V} \xrightarrow{\widehat{R}} \widehat{\mathfrak{g}}$ is a rRB Leibniz algebra. Moreover, this is an abelian extension of the rRB Leibniz algebra $V \xrightarrow{R} \mathfrak{g}$ by the $2$-term chain complex $W \xrightarrow{S} \mathfrak{h}$.

Let $(\alpha, \beta, \gamma)$ and $(\alpha', \beta', \gamma')$ be two cohomologous $2$-cocycles, say $(\alpha, \beta, \gamma) - (\alpha', \beta', \gamma') = \delta'_\mathrm{rRB} ((\kappa, \eta))$, for some $(\kappa, \eta) \in C^1_\mathrm{rRB} (V \xrightarrow{R} \mathfrak{g}; W \xrightarrow{S} \mathfrak{h})$. We define maps $\phi : \mathfrak{g} \oplus \mathfrak{h} \rightarrow \mathfrak{g} \oplus \mathfrak{h}$ and $\psi : V \oplus W \rightarrow V \oplus W$ by
\begin{align*}
\phi ((x, h)) = (x, h + \kappa (x) ) ~~~~ \text{ and } ~~~~ \psi ((v, w)) = (v, w + \eta (v) ).
\end{align*}
It is easy to see that $(\phi, \psi) : (\widehat{V} \xrightarrow{\widehat{R}} \widehat{\mathfrak{g}}) \rightsquigarrow (\widehat{V}' \xrightarrow{\widehat{R}'} \widehat{\mathfrak{g}}')$ is an isomorphism of abelian extensions. Hence there is a well-defined map
\begin{align*}
\Theta_2 : H^2_\mathrm{rRB} (V \xrightarrow{R} \mathfrak{g} ; W \xrightarrow{S} \mathfrak{h}) \rightarrow \mathrm{Ext}(V \xrightarrow{R} \mathfrak{g} ; W \xrightarrow{S} \mathfrak{h}).
\end{align*}
Finally, the maps $\Theta_1$ and $\Theta_2$ are inverses to each other. Hence the proof.
\end{proof}

\medskip

\noindent {\bf Further discussions.} Strongly homotopy Lie algebras ($L_\infty$-algebras) are a generalization of Lie algebras where the Jacobi identity holds up to certain homotopy. Other homotopy algebras, such as strongly homotopy associative algebras ($A_\infty$-algebras) and strongly homotopy Leibniz algebras ($Leibniz_\infty$-algebras) are also widely studied in the literature as they have close connections with topology, combinatorics, higher geometry and mathematical physics \cite{stas,nobert}. Representations of all these homotopy algebras are also well-known \cite{keller,lada-markl}. The notion of homotopy relative Rota-Baxter operators in the context of $L_\infty$-algebras (resp. $A_\infty$-algebras) are recently defined in \cite{laza-rota,DasSK}. Such operators are characterized by Maurer-Cartan elements in a suitable $L_\infty$-algebra constructed in the above papers. A triple consisting of an $L_\infty$-algebra (resp. $A_\infty$-algebra), a representation and a homotopy relative Rota-Baxter operator is called a homotopy relative Rota-Baxter Lie (resp. associative) algebra. In \cite{jiang-sheng} Jiang and Sheng showed that homotopy relative Rota-Baxter Lie algebras are closely related to the cohomology of relative Rota-Baxter Lie algebras introduced in \cite{laza-rota}.

In a forthcoming paper, we aim to define homotopy relative Rota-Baxter Leibniz algebras (homotopy rRB Leibniz algebras in short) by a similar construction given in \cite{laza-rota,DasSK}. It would be interesting to find the relation between homotopy rRB Leibniz algebras and the cohomology of rRB Leibniz algebras introduced in the present paper.

\medskip

An operad is a device that encodes a type of algebras \cite{lod-val-book}. Instead of studying the properties of a particular algebra, an operad helps to understand the universal operations that can be performed on the elements of any algebra of a given type. All classical algebras (such as Lie, associative, Leibniz, Poisson, pre-Lie etc.) can be described by suitable operads. In \cite{bala} Balavoine studied cohomology and deformation theory of $\mathcal{P}$-algebras, where $\mathcal{P}$ is any binary quadratic operad. The notion of Rota-Baxter operators and relative Rota-Baxter operators on $\mathcal{P}$-algebras are introduced, and their relation with the splitting of $\mathcal{P}$-algebras are obtained in \cite{bai-splitting}. Such notions generalize the standard relations among Rota-Baxter operators and dendriform algebras \cite{aguiar-pre}. In a future project, we study relative Rota-Baxter $\mathcal{P}$-algebras and develop their representations, cohomology, deformations and homotopy theory. We also expect that the operad governing homotopy Rota-Baxter $\mathcal{P}$-algebras is the minimal model of the operad of Rota-Baxter $\mathcal{P}$-algebras.

%\section{Homotopy rRB Leibniz algebras} 

\medskip

\noindent {\bf Acknowledgements.} The author would like to thank the esteemed referees for their useful comments on the earlier version of the manuscript. Some parts of the work was carried out when the author was a postdoctoral fellow at IIT Kanpur. The author also thanks IIT Kharagpur for the beautiful atmosphere where the final parts of the paper has been done.

\end{document}